\documentclass{article}
\usepackage{amsmath,amssymb,amsthm,float}

\title{Arithmetical completeness for some extensions of the pure logic of necessitation}

\author{Haruka Kogure\footnote{Email:kogure@stu.kobe-u.ac.jp}
\footnote{Graduate School of System Informatics, Kobe University, 1-1 Rokkodai, Nada, Kobe 657-8501, Japan.}}
\date{}
\theoremstyle{plain}
\newtheorem{thm}{Theorem}[section]
\newtheorem*{thm*}{Theorem}

\newtheorem{prop}[thm]{Proposition}
\newtheorem{cor}[thm]{Corollary}
\newtheorem{fact}[thm]{Fact}
\newtheorem*{fact*}{Fact}
\newtheorem{prob}[thm]{Problem}
\newtheorem*{prob*}{Problem}
\newtheorem{cl}{Claim}[section]
\newtheorem{scl}{Subclaim}

\theoremstyle{definition}
\newtheorem{defn}[thm]{Definition}

\newcommand{\PA}{\mathsf{PA}}
\newcommand{\PR}{\mathrm{Pr}}
\newcommand{\PL}{\mathsf{PL}}

\newcommand{\Prf}{\mathrm{Prf}}
\newcommand{\Prov}{\mathrm{Prov}}
\newcommand{\Proof}{\mathrm{Proof}}
\newcommand{\Con}{\mathrm{Con}}

\newcommand{\gn}[1]{\ulcorner#1\urcorner}
\newcommand{\D}[1]{\mathbf{D#1}}

\newcommand{\Fml}{\mathrm{Fml}_{\mathcal{L}_A}}
\newcommand{\num}{\overline}
\newcommand{\NA}{\mathbf{NA}}
\newcommand{\True}{\mathrm{True}}

\newcommand{\LA}{\mathcal{L}_A}
\newcommand{\Sub}{\mathsf{Sub}}
\newcommand{\MF}{\mathsf{MF}}
\newcommand{\N}{\mathbf{N}}
\newcommand{\tc}{\vdash^{t}}
\begin{document}

\maketitle

\begin{abstract}
We investigate the arithmetical completeness theorems 
of some extensions of Fitting, Marek, and Truszczy\'{n}ski's pure logic of necessitation $\N$.
For $m,n \in \omega$, let $\NA_{m,n}$, which was introduced by Kurahashi and Sato, be the logic obtained from $\N$ by adding the axiom scheme $\Box^n A \to \Box^m A$.
In this paper, among other things, we prove that for each $m,n \geq 1$, the logic $\NA_{m,n}$ becomes a provability logic, that is, there exists a provability predicate $\PR_T(x)$ of $T$ whose $T$-verifiable modal principles are exactly the logic $\NA_{m,n}$.
\end{abstract}
%we clarify the situations that for each $m,n \in \omega$, the logic $\NA_{m,n}$, which was introduced by Kurahashi and Sato, becomes a provability logic of a $\Sigma_1$ provability predicate $\PR_T(x)$.
%We prove that for each $m,n \geq 1$, the logic $\NA_{m,n}$ is a provability logic of some $\Sigma_1$ provability predicate
%and that for each $m \geq 1$,  so is $\NA_{m,0}$ with respect to $\Sigma_1$-ill theories.
%In addition, we show that $\NA_{0,n}$ for $n \geq 1$ is not a provability logic for any $\PR_T(x)$ and if $T$ is $\Sigma_1$-sound, then $\NA_{m,0}$ for $m \geq 1$ is not a provability logic for any $\PR_T(x)$. 

\section{Introduction}\label{sec1}
Let $T$ denote a primitive recursively axiomatized consistent extension of Peano Arithmetic $\PA$.
We say that a formula $\PR_T(x)$ is a provability predicate of $T$ 
if $\PR_T(x)$ weakly represents the set of all theorems of $T$ in $\PA$, 
that is, for any formula $\varphi$, $T \vdash \varphi$ if and only if $\PA \vdash \PR_T(\gn{\varphi})$. 
Here, $\gn{\varphi}$ is the numeral of the G\"{o}del number of $\varphi$.
In a proof of the second incompleteness theorem, it is essential to construct a canonical $\Sigma_1$ provability predicate $\Prov_T(x)$ of $T$ 
satisfying the following conditions:
\begin{enumerate}
\item[$\D{2}$:]
$T \vdash \Prov_T(\gn{\varphi \to \psi}) \to (\Prov_T(\gn{\varphi}) \to \Prov_T(\gn{\psi}))$.
\item[$\D{3}$:]
$T \vdash \Prov_T(\gn{\varphi}) \to \Prov_T(\gn{\Prov_T(\gn{\varphi})})$.
\end{enumerate}
A provability predicate $\PR_T(x)$ can be considered as a modality. 
For example, the schemes
$\Box (A \to B) \to (\Box A \to \Box B)$ and $\Box A \to \Box \Box A$ are modal counterparts of $\D{2}$ and $\D{3}$ respectively.
We say that an \textit{arithmetical interpretation $f$ based on $\PR_T(x)$} is a mapping from  modal formulas to arithmetical sentences 
such that $f$ preserves 
 propositional connectives, and $f(\Box A)$ is $\PR_T(\gn{f(A)})$. 
For each provability predicate $\PR_T(x)$ of $T$,
let $\PL(\PR_T)$ denote the set of all modal formulas $A$ satisfying $T \vdash f(A)$ for all arithmetical interpretations $f$ based on $\PR_T(x)$.
The set $\PL(\PR_T)$ is called the \textit{provability logic} of $\PR_T(x)$.
A well-known result for the study of provability logics is Solovay's arithmetical completeness theorem \cite{sol}
saying that if $T$ is $\Sigma_1$-sound, then $\PL(\Prov_T)$ is exactly the modal logic $\mathbf{GL}$.
For each non-standard provability predicate, which is not the canonical provability predicate $\Prov_T(x)$, its provability logic may not be $\mathbf{GL}$, and there has been recent research on provability logics in this direction.
One goal of studies in this direction is to obtain an in-depth understanding of the following problem.

  \begin{prob}[{\cite[Problem 2.2]{Kur23}}]\label{prob2}
	For which modal logic $L$ is there a provability predicate $\PR_T(x)$
	such that $L = \PL(\PR_T)$?
	\end{prob}
This problem is exactly the question of which modal logics are provability logics.
Kurahashi \cite{Kur18-2, Kur23, Kur20} proved that modal logics such as $\mathbf{K}$ and $\mathbf{KD}$ are provability logics,
and Kogure and Kurahashi \cite{KK} proved that so are non-normal modal logics such as $\mathbf{MN}$ and $\mathbf{MN4}$.
Also, not all modal logics can be a provability logic of some provability predicate.
It is known that each of modal logics $\mathbf{KT}$, $\mathbf{K4}$, $\mathbf{KB}$, and $\mathbf{K5}$ is not a provability logic of any provability predicate (cf.~\cite{MON, Kur18}).
%Here, a Rosser provability predicate $\PR_{T}^{R}(x)$ is a provability predicate naturally expressing that ``there exists a $T$-proof $y$ of $x$ such that there is no $T$-proof of the negation of $x$ less than $y$''.  
%Provability predicates corresponding to $\mathbf{K}$ and $\mathbf{KD}$ are non-standard and
%provability logics of non-standard provability predicates are studied in \cite{Vis, Sha}.

Fitting, Marek, and Truszczy\'{n}ski \cite{fmt} introduced the pure logic of necessitation $\mathbf{N}$, which is obtained by removing the distribution axiom $\Box (A \to B) \to (\Box A \to \Box B)$ from the basic normal modal logic $\mathbf{K}$,
and Kurahashi proved that $\N$ is a provability logic of some provability predicate. 
Kurahashi \cite{Kur23} also introduced an extension $\mathbf{N4}$ of $\N$ and proved that $\mathbf{N4}$ has the finite frame property.
Here, a modal logic $L$ is said to have the finite frame property if for any formula $A$, $A$ is provable in $L$ if and only if $A$ is valid in all finite frames validating all theorems of $L$.
By using the finite frame property of $\mathbf{N4}$, he also proved that $\mathbf{N4}$ is a provability logic of some provability predicate. 
%and so are extensions $\mathbf{N4}$ and $\mathbf{NR}$
%of $\N$. 
%Here, $\mathbf{N4}$ and $\mathbf{NR}$ are obtained by adding the axiom scheme $\Box A \to \Box \Box A$ and the inference rure $\dfrac{\neg A}{\neg \Box A}$ to $\N$ respectively.
The finite frame property of $\mathbf{N4}$ was generalized by Kurahashi and Sato \cite{KS}.
For each pair of natural numbers $m$ and $n$, 
Kurahashi and Sato introduced the extension  $\NA_{m,n}$ obtained from $\N$ by adding
 the axiom scheme $\Box^n A \to \Box ^m A$,
and proved that if  $m \geq 1$ or $n \leq 1$, then $\NA_{m,n}$ has the finite frame property.

In the viewpoint of Problem \ref{prob2}, it is natural to ask the following question:
``For each $m,n \in \omega$, can the logic $\NA_{m,n}$ be a provability logic?''
In this paper, we clarify the situation on this problem.
If $m=n$, then $\NA_{m,n}$ is exactly $\N$, and Kurahashi proved that $\N$ is a provability logic of some $\Sigma_1$ provability predicate, so it suffices to consider the case $m \neq n$.
The following table summarizes the situation when $m \neq n$.
% \begin{table}[H]\label{tab}
% \centering
% \caption{Is there a $\Sigma_1$ $\PR_T(x)$ of $T$ such that $\PL(\PR_T)= \NA_{m,n}$?}
% \scriptsize{
% \begin{tabular}{|c|c|c|c|}
% \hline
%  & $m,n \geq 1$  & $m=0$ \& $n \geq 1$ & $m \geq 1$ \& $n=0$ \\ \hline
%  $T$ is $\Sigma_1$-sound & Yes (Theorems \ref{thm4-4}, \ref{thm4-5}) &No (Proposition \ref{not arith1}) & No (Proposition \ref{not arith2}) \\ \hline
%  $T$ is $\Sigma_1$-ill &  Yes (Theorems \ref{thm4-5}, \ref{thm4-6}) &No (Proposition \ref{not arith1})& Yes (Theorem \ref{thm4-7}) \\ \hline
% \end{tabular}
% }
% \end{table}
\begin{table}[H]\label{tab}
\centering
\caption{Is there a $\Sigma_1$ $\PR_T(x)$ of $T$ such that $\PL(\PR_T)= \NA_{m,n}$?}
\begin{tabular}{|c|c|c|}
\hline
 & $T$ is $\Sigma_1$-sound & $T$ is $\Sigma_1$-ill \\ \hline
  $m,n \geq 1$ & Yes (Theorems \ref{thm4-4}, \ref{thm4-5}) & Yes (Theorems \ref{thm4-5}, \ref{thm4-6}) \\ \hline
$m=0 \ \& \ n \geq 1$ & No (Proposition \ref{not arith1}) & No (Proposition \ref{not arith1}) \\ \hline
$m \geq 1$ \& $n=0$ & No (Proposition \ref{not arith2})  & Yes (Theorem \ref{thm4-7})
   \\ \hline
\end{tabular}
\end{table}

Here, a theory $T$ is $\Sigma_1$-ill if it is not $\Sigma_1$-sound.
The following shows the details of the theorems and propositions in the table.
\begin{itemize}
\item 
If $n > m  \geq 1$ and $T$ is $\Sigma_1$-sound, then $\NA_{m,n}$ is a provability logic of some $\Sigma_1$ provability predicate (Theorem \ref{thm4-4}).
\item 
If $m > n \geq 1$, then $\NA_{m,n}$ is a provability logic of some $\Sigma_1$ provability predicate (Theorem \ref{thm4-5}).
\item 
If $n > m  \geq 1$ and $T$ is $\Sigma_1$-ill, then $\NA_{m,n}$ is a provability logic of some $\Sigma_1$ provability predicate (Theorem \ref{thm4-6}).
\item 
If $m \geq 1$ and $T$ is $\Sigma_1$-ill, then $\NA_{m,0}$ is a provability logic of some $\Sigma_1$ provability predicate (Theorem \ref{thm4-7}).
\item 
If $n \geq 1$, then $\NA_{0,n}$ is not a provability logic of any provability predicate (Proposition \ref{not arith1}).
\item 
If $T$ is $\Sigma_1$-sound and $m \geq 1$, then $\NA_{m,0}$ is not a provability logic of any $\Sigma_1$ provability predicate (Proposition \ref{not arith2}).
\end{itemize}
Moreover, it was recently clarified by Kurahashi  \cite{Kur25} that the principle $\Box^{n} A \rightarrow \Box^{m} A$ is closely related to the second incompleteness theorem. 
In his study, the condition $\mathbf{D3}^{n}_{m} : T \vdash \PR_T^{n}(\gn{\varphi}) \rightarrow \PR_T^{m}(\gn{\varphi})$ is analyzed from the viewpoint of the second incompleteness theorem, and it is shown that $\PR_{T}(x)$ satisfying $\mathbf{D3}^{n}_{m}$, where $m> n \geq 1$, together with several additional conditions yields the second incompleteness theorem.

The organization of this paper is as follows. 
Section \ref{sec2} presents the preliminaries and background. 
Section \ref{sec3} states the main results of this paper.
Section \ref{sec4} and \ref{sec5} are devoted to the proofs of our results.
Finally, Section \ref{sec6} provides the conclusion of this paper. In particular, we discuss the relationship between our results on $\NA_{m,n}$ and
recent study by Kurahashi \cite{Kur25}. 

\section{Preliminaries and Background}\label{sec2}
For general background on arithmetic and the incompleteness theorems, see e.g.~\cite{HP93,Lin}.
Throughout this paper, let $T$ denote a primitive recursively axiomatized consistent extension
of Peano arithmetic $\PA$ in the language $\LA$ of first-order arithmetic.
We say that $T$ is $\Sigma_1$-ill if $T$ is not $\Sigma_1$-sound.
Let $\omega$ be the set of all natural numbers. For each $n \in \omega$,  $\num{n}$ denotes the numeral of $n$.
In the present paper, we fix a natural G\"{o}del numbering such that if $\alpha$ is a proper sub-expression of a finite sequence $\beta$ of $\LA$-symbols, 
then the G\"{o}del number of $\alpha$ is less than that of $\beta$. 
For each formula $\varphi$, let $\gn{\varphi}$ be the numeral of the G\"{o}del number of $\varphi$.
Let $\{ \xi_t \}_{t \in \omega}$ denote the repetition-free primitive recursive enumeration of all $\LA$-formulas in ascending order of G\"{o}del numbers. 
We note that if $\xi_u$ is a proper subformula of $\xi_v$, then $u<v$.

The classes of formulas $\Delta_0$, $\Sigma_1$, and $\Pi_1$ are defined, as usual, as the primitive recursive classes of $\Delta_0$, $\Sigma_1$, and $\Pi_1$ formulas, respectively (see \cite{HP93}).
The primitive recursiveness of these classes will be used in Section \ref{4.1}.
%注意としてロジカルなもので閉じると計算可能でないので，そこは区別する．
% Let $\Delta_0$ be the set of all formulas whose every quantifier is bounded.
% The classes $\Sigma_1$ and $\Pi_1$ are defined as the smallest classes satisfying the following conditions:
% \begin{itemize}
% \item 
% $\Delta_0 \subseteq \Sigma_1 \cap \Pi_1$.
% \item 
% $\Sigma_1$ (resp.~$\Pi_1$) is closed under conjunction, disjunction and existential (resp. universal) quantification.
% \item 
% If $\varphi$ is $\Sigma_1$ (resp.~$\Pi_1$), then $\neg \varphi$ is $\Pi_1$ (resp.~$\Sigma_1$).
% \item 
% If $\varphi$ is $\Sigma_1$ (resp.~$\Pi_1$) and $\psi$ is $\Pi_1$ (resp.~$\Sigma_1$), then $\varphi \to \psi$ is $\Pi_1$ (resp.~$\Sigma_1$).
% \end{itemize}
% The classes $\Delta_0$, $\Sigma_1$ and $\Pi_1$ are not closed under logical equivalence, and these classes are primitive recursive.
We say that a formula $\varphi$ is $\Delta_1(\PA)$ if $\varphi$ is $\Sigma_1$ and $\varphi$ is $\PA$-equivalent to some $\Pi_1$ formula.
Let $\Fml(x)$ and $\Sigma_1(x)$ be $\Delta_1(\PA)$ formulas naturally expressing that ``$x$ is an $\LA$ formula'' and ``$x$ is a $\Sigma_1$ sentence'', respectively. 
Let $\True_{\Sigma_1}(x)$ be a $\Sigma_1$ formula naturally saying that ``$x$ is a true $\Sigma_1$ sentence'' and
we may assume that $\True_{\Sigma_1}(x)$ is of the form $\exists y \eta (x,y)$ for some $\Delta_0$ formula $\eta(x,y)$.
The formula $\True_{\Sigma_1}(x)$ satisfies the condition that for any $\Sigma_1$ sentence $\varphi$, $\PA \vdash \varphi \leftrightarrow \True_{\Sigma_1}(\gn{\varphi})$
(see~\cite{HP93, Kay91}). 
\subsection{Provability predicate}
We say that a formula $\PR_T(x)$ is a \textit{provability predicate} of $T$ 
if for any formula $\varphi$, $T \vdash \varphi$ if and only if $\PA \vdash \PR_T(\gn{\varphi})$. Note that the definition of provability predicates is not restricted to $\Sigma_1$ formula
so that we can also discuss provability predicates that are not $\Sigma_1$ (see Problem~\ref{prob3}).
Let $\mathbb{N}$ be the standard model of arithmetic.
We note that if $\PR_T(x)$ is a $\Sigma_1$ formula, 
then the condition that $\PR_T(x)$ is a provability predicate of $T$ is equivalent to the condition that for any formula $\varphi$, $T \vdash \varphi$  if and only if $\mathbb{N} \models \PR_T(\gn{\varphi})$.
Also, to ensure that provability predicates (not necessarily $\Sigma_1$) satisfy the condition $\mathbf{D1}$,
where $\mathbf{D1}$ is the requirement that for any formula $\varphi$, 
$T \vdash \varphi$ implies $T \vdash \PR_T(\gn{\varphi})$, we adopt $\PA$-based definition of provability predicates.

For each formula $\tau(v)$ representing the set of all axioms of $T$ in $\PA$,
we can construct a formula $\Prf_{\tau}(x,y)$ naturally expressing that 
``$y$ is the G\"{o}del number of a proof of $x$ from the set of all sentences satisfying $\tau(v)$'',
 \footnote{In some textbook (e.g. \cite{HP93}), $\Prf_{\tau}(y,x)$ expresses that ``$y$ is the G\"{o}del number of a proof of $x$ from the set of all sentences satisfying $\tau(v)$.''  
 Our convention for $\Prf_{\tau}(x,y)$ is the reverse, following \cite{FEf,Lin}.}
and we can define a provability predicate $\PR_{\tau}(x)$ as $\exists y \Prf_{\tau}(x,y)$ (see~\cite{FEf,Lin}). 
Throughout the paper, when we say that 
$p$ is a proof, we always mean that $p$ is the Gödel number coding that proof.
If $\tau(v)$ is a $\Delta_1(\PA)$ formula, then we can define $\Prf_{\tau}(x,y)$ as a $\Delta_1(\PA)$ formula.
Let $\Con_{\tau}$ denote $\neg \PR_{\tau}(\gn{0=1})$.
In this paper, we assume that $\Prf_{\tau}(x,y)$ is a single-conclusion, 
that is, $\PA \vdash \forall y \forall x_1 \forall x_2 (\Prf_{\tau}(x_1,y) \wedge \Prf_{\tau}(x_2,y) \to x_1=x_2)$ holds.
We note that $\Prf_{\tau}(x)$ satisfies $\PA \vdash \forall x \bigl( \PR_{\tau}(x) \to \forall z \exists y>z \Prf_{\tau}(x,y) \bigr)$.
In addition, $\PR_{\tau}(x)$ satisfies the following conditions (see~\cite{FEf,Lin}):
\begin{itemize}
\item 
$\PA \vdash \forall x \forall y \bigl(\PR_{\tau}(x \dot{\to} y) \to \bigl(\PR_{\tau}(x) \to \PR_{\tau}(y)\bigr)\bigr)$.
\item 
If $\sigma(x)$ is a $\Sigma_1$ formula, then $\PA \vdash \forall x \bigl( \sigma (x) \to \PR_{\tau}(\gn{\sigma(\dot{x})}) \bigr)$.
\end{itemize}
Here, $x \dot{\to} y$ and $\gn{\sigma(\dot{x})}$ are  primitive recursive terms 
corresponding to  primitive recursive functions calculating the G\"{o}del number of $\varphi \to \psi$ from those of $\varphi$ and $\psi$ 
and calculating the G\"{o}del number of $\sigma(\num{n})$ from $n \in \omega$, respectively.
The first condition is a stronger version of the condition $\D{2}$.
If $\PR_{\tau}(x)$ is $\Sigma_1$, then $\PR_{\tau}(x)$ satisfies the condition $\D{3}$.
Here, the conditions $\D{2}$ and $\D{3}$ are introduced in Section \ref{sec1}.
From now on, we fix a $\Delta_1(\PA)$ formula $\tau(x)$ representing $T$ in $\PA$.
Let $\Proof_T(x,y)$ and $\Prov_T(x)$ be the $\Delta_1(\PA)$ formula $\Prf_{\tau}(x,y)$ and the $\Sigma_1$ provability predicate $\PR_{\tau}(x)$ respectively. Let $\Con_T$ denote $\neg \Prov_T(\gn{0=1})$.

\subsection{Provability logic}
The language of modal propositional logic $\mathcal{L}_{\Box}$ consists of  
propositional variables, the logical constant $\bot$, the logical connectives $\neg, \wedge, \vee, \to$,
and the modal operator $\Box$.
Let $\MF$ be the set of all $\mathcal{L}_{\Box}$-formulas.
For each $A \in \MF$, let $\Sub(A)$ denote the set of all subformulas of $A$.
For each $A \in \MF$ and $n \in \omega$, we define  $\Box^n A$ inductively as follows:
Let $\Box^0 A$ be $A$ and let $\Box^{n+1} A$ be $\Box \Box^n A$.
The axioms of the basic modal logic $\mathbf{K}$ consists of propositional tautologies in the language  $\mathcal{L}_\Box$
and the distribution axiom scheme $\Box (A \to B) \to (\Box A \to \Box B)$. 
The inference rules of modal logic $\mathbf{K}$ consists of Modus Ponens $(\mathrm{MP}) \ \dfrac{A \quad A \to B}{B}$
and Necessitation $(\mathrm{Nec}) \ \dfrac{A}{\Box A}$.
%The logic $\mathbf{K4}$ is obtained by adding the axiom scheme $\Box A \to \Box  \Box A$ to $\mathbf{K}$.
%The schemes $\Box (A \to B) \to (\Box A \to \Box B)$ and $\Box A \to \Box  \Box A$ are modal counter parts $\D{2}$ and $\D{3}$ respectively.

For each provability predicate $\PR_T(x)$ of $T$,
we say that a mapping $f$ from $\MF$ to a set of $\LA$-sentences 
is an \textit{arithmetical interpretation} based on $\PR_T(x)$ if $f$ satisfies the following clauses:
\begin{itemize}
\item 
$f(\bot)$ is $0=1$, 
\item
$f(\neg A)$ is $\neg f(A)$,
\item 
$f(A \circ B)$ is $f(A) \circ f(B)$ for $\circ \in \{ \wedge , \vee , \to \}$, and
\item 
$f(\Box A)$ is $\PR_T(\gn{f(A)})$.
\end{itemize}
For any provability predicate $\PR_T(x)$ of $T$, 
let $\PL(\PR_T)$ denote the set of all $\mathcal{L}_{\Box}$-formulas $A$ such that 
for any arithmetical interpretation $f$ based on $\PR_T(x)$, $T \vdash f(A)$.
We call the set $\PL(\PR_T)$ \textit{provability logic} of $\PR_T(x)$ (For an introduction to provability logic, see \cite{Bool,Smo}).
%Kurahashi \cite{Kur23} states that we can approach the study of provability logics 
%from two directions as follows. 
%\begin{prob}[{\cite[Problem 2.1]{Kur23}}]
%For each provability predicate $\PR_T(x)$ of $T$, 
%how is $\PL(\PR_T)$ axiomatized and what does it have? 
%\end{prob}
%\begin{prob}[{\cite[Problem 2.2]{Kur23}}]
%For which modal logic $L$ is there a provability predicate $\PR_T(x)$
%such that $L = \PL(\PR_T)$?
%\end{prob}

A well-known result in the study of provability logics is Solovay's arithmetical completeness theorem.
The modal logic $\mathbf{GL}$ is obtained by adding axiom scheme $\Box (\Box A \to A) \to \Box A$ to $\mathbf{K}$. 
\begin{thm}[Solovay \cite{sol}]
If $T$ is $\Sigma_1$-sound, then $\PL(\Prov_T) = \mathbf{GL}$.
\end{thm}

Fitting, Marek, and Truszczy\'{n}ski \cite{fmt} introduced the pure logic of necessitation $\mathbf{N}$,
which is obtained by removing the distribution axiom $\Box (A \to B) \to (\Box A \to \Box B)$ from $\mathbf{K}$.
%In this paper, we investigate some extensions of $\mathbf{N}$ concering the second problem.
In Kurahashi \cite{Kur23}, it is proved that some extensions of $\mathbf{N}$ can be a provability logic of a provability predicate $\PR_T(x)$ of $T$.
The modal logic $\mathbf{N4}$ is obtained by adding the axiom scheme
$\Box A \to \Box \Box A$.
The following are obtained in \cite{Kur23}.
%The  modal logics $\mathbf{N4}$, $\mathbf{NR}$ and $\mathbf{NR4}$ are all introduced in Kurahashi \cite{Kur23}.
\begin{itemize}
\item 
For each $L \in \{ \mathbf{N}, \mathbf{N4} \}$,
there exists a $\Sigma_1$ provability predicate $\PR_T(x)$ of $T$
such that $\PL(\PR_T)= L$.
\item 	
$\N = \bigcap \{\PL(\PR_T) \mid \PR_T(x)$  is a provability predicate of $T\}$.
\item 
$\mathbf{N4} = \bigcap \{\PL(\PR_T) \mid \PR_T(x)$  is a provability predicate of $T$ satisfying $\D{3} \}$.

\end{itemize}

\subsection{The logic $\NA_{m,n}$}
Fitting, Marek, and Truszczy\'{n}ski introduced a relational semantics for $\mathbf{N}$.
\begin{defn}[$\N$-frames and $\N$-models]
\leavevmode

\begin{itemize}
\item 
A tuple $(W, \{ \prec_B \}_{B \in \MF} )$ is called an $\N$-\textit{frame}
if $W$ is non-empty set and $\prec_B$ is a binary relation on $W$ for every $B \in \MF$. 
\item 
An $\N$-frame $(W, \{ \prec_B \}_{B \in \MF} )$ is finite if $W$ is finite.
\item 
A triple $(W, \{ \prec_B \}_{B \in \MF}, \Vdash)$ is called an $\N$-\textit{model}
if $(W, \{ \prec_B \}_{B \in \MF} )$ is an $\N$-frame and $\Vdash$ is a binary relation
between $W$ and $\MF$ satisfying the usual conditions for propositional connectives and the following condition:
\[
x \Vdash \Box B \iff \forall y \in W (x \prec_B y \Longrightarrow y \Vdash B).
\]
\item 
A formula $A$ is \textit{valid} in an $\N$-model $(W, \{ \prec_B \}_{B \in \MF}, \Vdash)$
if $x \Vdash A$ for every $x \in W$.
\item 
A formula $A$ is \textit{valid} in an $\N$-frame $ \mathcal{F} = (W, \{ \prec_B \}_{B \in \MF})$
if $A$ is valid in any $\N$-model $(\mathcal{F}, \Vdash)$ based on $\mathcal{F}$.
\end{itemize}
\end{defn}
Note that the main difference from the usual Kripke semantics is that each formula has an accessibility relation.
Fitting, Marek, and Truszczy\'{n}ski proved that
$\N$ is sound and complete with respect to the above semantics.
In addition, $\N$ has the finite frame property with respect to this semantics.
\begin{thm}[{\cite[Theorem 3.6 and Theorem 4.10]{fmt}}]
For any $\mathcal{L}_\Box$-formula $A$, the following are equivalent:
\begin{enumerate}
\item 
$\N \vdash A$.
\item 
$A$ is valid in all $\N$-frames.
\item
$A$ is valid in all finite $\N$-frames.
\end{enumerate}
\end{thm}
%Though each $\N$-frame has infinitely many binary relations $\{ \prec_B\}_{B \in \MF}$,
%to determine the truth of a single formula $A$, it suffices to consider only a finite subset of $\{ \prec_B\}_{B \in \MF}$.
%\begin{prop}[{\cite[Theorem 4.11]{fmt}}]
%Let $A \in \MF$. Let $(W, \{ \prec_B\}_{B \in \MF}, \Vdash)$ and $(W, \{ \prec_{B}'\}_{B \in \MF}, \Vdash')$ be 
%$\N$-models satisfying $\prec_B = \prec_{B}'$ for every $\Box B \in \Sub(A)$.
%Then, for any $x \in W$, $x \Vdash A $ if and only if $x \Vdash' A$.
%\end{prop}

For each $m, n \in \omega$, the logic $\N \mathbf{A}_{m,n}$ is obtained by adding the axiom scheme 
\[
\mathrm{Acc}_{m,n} : \Box^n A \to \Box ^m A
\]
to $\N$, and the logic $\mathbf{N}^+ \mathbf{A}_{m,n}$ is obtained by adding the inference rule
\[
(\mathrm{Ros}^{\Box}) \dfrac{\neg \Box A}{\neg \Box \Box A}
\]
to $\N \mathbf{A}_{m,n}$.
Here, the rule $\mathrm{Ros}^{\Box}$ is a weak variant of the rule $(\mathrm{Ros}) \dfrac{\neg A}{\neg \Box A}$, originally introduced in \cite{Kur23} to investigate Rosser provability predicates.
Extensions $\N \mathbf{A}_{m,n}$ and $\mathbf{N}^+ \mathbf{A}_{m,n}$ of $\N$ are introduced in Kurahashi and Sato \cite{KS}.
It is known that $\NA_{m,n}$ is exactly $\mathbf{N}^+ \mathbf{A}_{m,n}$ for $m \geq 1$ or $n \leq 1$.
\begin{prop}[{\cite[Proposition 4.2]{KS}}]
Let $m \geq 1$ or $n \leq 1$.
The rule $\mathrm{Ros}^\Box$ is admissible in $\N \mathbf{A}_{m,n}$, that is, adding the rule $\mathrm{Ros}^\Box$ to $\N \mathbf{A}_{m,n}$ does not yield any new theorems.
Consequently, $\NA_{m,n}$ is exactly $\N^+ \mathbf{A}_{m,n}$.
\end{prop}

Kurahashi and Sato also introduced the notion $(m,n)$-\textit{accessibility}.
\begin{defn}
Let $\mathcal{F}= (W, \{ \prec_B \}_{B \in \MF})$ be an $\N$-frame.
\leavevmode

\begin{itemize}
\item 
Let $x, y \in W$, $B \in \MF$ and $k \in \omega$. We define $x \prec_{B}^k y$ as follows:
\[x \prec_{B}^k y : \iff \begin{cases} 
x=y & \text{if} \ k=0 \\
\exists w \in W( x \prec_{\Box^{k-1} B} w \prec_{B}^{k-1} y ) & \text{if} \ k \geq 1.
\end{cases}
\]

\item 
We say that the frame $\mathcal{F}$ is $(m,n)$-\textit{accessible} 
if  for any $x, y \in W$ and $B \in \MF$, if $x \prec_{B}^m y$, then $x \prec_{B}^n y$.
%\item 
%Let $\Gamma \subseteq \MF$.
%We say that the frame $\mathcal{F}$ is $\Gamma$-$(m,n)$-\textit{accessible} 
%if $\mathcal{F}$ is $B$-$(m,n)$-accessible for every $\Box^m B \in \Gamma$.
%\item 
%We say that the frame $\mathcal{F}$ is $(m,n)$-\textit{accessible} if $\mathcal{F}$ is $\MF$-$(m,n)$-accessible.
\end{itemize}
\end{defn}
%In particular, $(W, \{\prec_{B}\}_{B \in \MF})$ is $(2,1)$-accessible is that 
%for each $x,y, z \in W$ and $B \in \MF$, if $x \prec_{\Box B} y$ and $y \prec_B z$, then $x \prec_B z$.
%The notion $(m,n)$-accessibility is a generalization of that of transitivity. 
%It is known that $\N \mathbf{4}$ is sound and complete with respect to $(2,1)$-accessible $\N$-frames (Kurahashi \cite[Theorem 3.13]{Kur23}).
%We predict that $\N \mathbf{A}_{m,n}$ is sound and complete with respect to $(m,n)$-accessible $\N$
%-frames, but for $n \geq 2$, $\N \mathbf{A}_{0,n}$ is incomplete with respect to $(0,n)$-accessible $\N$-frames (\cite[Corrolary 4.6]{KS}).

For $m,n \in \omega$, the logic $\mathbf{N}^+ \mathbf{A}_{m,n}$ is sound and complete and has the finite frame property with respect to $(m,n)$-accessible $\N$-frames.
\begin{fact}[The finite frame property of $\mathbf{N}^+ \mathbf{A}_{m,n}$ {\cite[Theorem 5.1]{KS}}] \label{fact2-7}
Let $A \in \MF$ and $m,n \in \omega$. The following are equivalent:
\begin{enumerate}
\item 
$\mathbf{N}^+ \mathbf{A}_{m,n} \vdash A$.
\item 
$A$ is valid in all $(m,n)$-accessible $\N$-frames.
\item 	
$A$ is valid in all finite $(m,n)$-accessible $\N$-frames.

\end{enumerate}
	
\end{fact}

\begin{fact}[The decidability of $\mathbf{N}^+ \mathbf{A}_{m,n}$ {\cite[Corollary 6.10]{KS}}] \label{fact2-8}
For each $m, n \in \omega$, the logic $\mathbf{N}^+ \mathbf{A}_{m,n}$ is decidable. In particular, if $m \geq 1$ or $n \leq 1$, the logic $\NA_{m,n}$ is decidable.
\end{fact}
From the proof of {\cite[Theorem 5.1]{KS}}, it is obvious that we can primitive recursively construct a finite $(m,n)$-accessible $\N$-model which falsifies a formula $A$ such that $\mathbf{N}^+ \mathbf{A}_{m,n} \nvdash A$.
%We can see that for each $m,n \in \omega$, $\NA_{m,n}$ is primitive recursively determined from the proof of {\cite[Theorem 5.1]{KS}}.

\section{Main results}\label{sec3}
In this section, we present our results and describe the organization of the following sections.
Our results are stated as follows:
%In this paper, we mainly prove arithmetical completeness theorems of $\NA_{m,n}$.
%In this paper, we prove the following two statements holds:
\begin{thm}\label{main}
Suppose $m \neq n$.
\leavevmode

\begin{itemize}
\item 
For each $m, n \geq 1$, there exists a $\Sigma_1$ provability predicate $\PR_T(x)$ such that $\NA_{m,n} = \PL(\PR_T)$.
\item 
Suppose $T$ is $\Sigma_1$-ill. For each $m \geq 1$, there exists a $\Sigma_1$ provability predicate $\PR_T(x)$ such that  $\NA_{m,0} = \PL(\PR_T)$.
\item 
For each $n \geq 1$, 
there exists no provability predicate $\PR_T(x)$ of $T$ such that  
$\NA_{0,n}  \subseteq \PL(\PR_T)$.
\item 
Suppose $T$ is $\Sigma_1$-sound. For each $m \geq 1$, 
there exists no $\Sigma_1$ provability predicate $\PR_T(x)$ of $T$ such that  
$\NA_{m,0}  \subseteq \PL(\PR_T)$.

\end{itemize}
\end{thm}
Table \ref{tab} in Section \ref{sec1} summarizes the situation in Theorem \ref{main}.
Our strategy to prove the first clause of Theorem \ref{main} depends on whether $T$ is $\Sigma_1$-sound or not.
In fact, we prove the following theorems.
%理論TがΣ_1健全の時とそうでないときの二種類に分けて，NA_m,nの算術的完全性定理を証明する．
%このようにする理由は，TがΣ_1 健全かどうかで，n>m\geq1 のときの証明が変わってくるということ，特にTがΣ_1健全でないときは
%n=0についても算術的完全になるということから，このようにΣ_1 健全かどうかで分ける．
\begin{thm}\label{thm4-2}
	Suppose $m \neq n$.
	If $T$ is $\Sigma_1$-sound and $m, n \geq 1$, then there exists a $\Sigma_1$ provability
	predicate $\PR_T(x)$ of $T$ such that
	\begin{enumerate} 
	\item 
	for any $A \in \mathsf{MF}$ and arithmetical interpretation $f$ based on $\PR_T(x)$,
	if $\mathbf{N}\mathbf{A}_{m,n} \vdash A$, then $T \vdash f(A)$, and 
	\item 
	there exists an arithmetical interpretation $f$ based on $\PR_T(x)$ such that 
	$\mathbf{N}\mathbf{A}_{m,n} \vdash A$ if and only if $T \vdash f(A)$.
	\end{enumerate}
	
	\end{thm}
	\begin{thm}\label{thm4-3}
		Suppose $m \neq n$.
		If $T$ is $\Sigma_1$-ill, $m \geq 1$ and $n \geq 0$, then there exists a $\Sigma_1$ provability
		predicate $\PR_T(x)$ of $T$ such that
		\begin{enumerate} 
		\item 
		for any $A \in \mathsf{MF}$ and arithmetical interpretation $f$ based on $\PR_T(x)$,
		if $\mathbf{N}\mathbf{A}_{m,n} \vdash A$, then $T \vdash f(A)$, and 
		\item 
		there exists an arithmetical interpretation $f$ based on $\PR_T(x)$ such that 
		$\mathbf{N}\mathbf{A}_{m,n} \vdash A$ if and only if $T \vdash f(A)$.
		\end{enumerate}
		
		\end{thm}

Theorems \ref{thm4-2} and \ref{thm4-3} imply the first and second clauses of Theorem \ref{main}.	
In Section \ref{sec4}, we prove Theorem \ref{thm4-2} by distinguishing the two more cases $n > m \geq 1$ and $m > n \geq 1$ and Theorem \ref{thm4-3} with respect to the case $m>n \geq 1$.
In Section \ref{sec5}, we prove  Theorem \ref{thm4-3} with respect to the cases $n>m \geq 1$ and $m > n=0$.

First, we show the third clause of Theorem \ref{main}.
For any $n \in \omega$, provability predicate $\PR_T(x)$, and formula $\varphi$,  we define $\PR_{T}^n(\gn{\varphi})$ inductively as follows:
Let $\PR_{T}^0 (\gn{\varphi})$ be $\varphi$ and $\PR_{T}^{n+1} (\gn{\varphi})$ be $\PR_{T}(\gn{\PR_{T}^n(\gn{\varphi})})$.
The following proposition is a generalization of the fact $\PL(\PR_T) \nsupseteq \mathbf{KT} \ (= \mathbf{K} + \Box A \to A)$ for any $\PR_T(x)$ (cf.~\cite{MON}).
%modal logics containing $\Box^n p \to p$, where $n \geq 1$. 
\begin{prop}\label{not arith1}
For each $n \geq 1$, there exists no provability predicate $\PR_T(x)$ of $T$ 
such that $\Box^n p \to p \in \PL (\PR_T)$.
\end{prop}
\begin{proof}
Suppose that there is a provability predicate $\PR_T(x)$ of $T$ satisfying
$\Box^n p \to p \in \PL (\PR_T)$. 
By the fixed-point lemma (see \cite{Lin}), we have a sentence $\sigma$ such that $\PA \vdash \neg 
\PR_{T}^n (\gn{\sigma}) \leftrightarrow \sigma$. 
Let $f$ be an arithmetical interpretation satisfying $f(p) = \sigma$.
 Then, it follows that
$T \vdash f(\Box^n p) \to f(p)$, that is, we obtain $T \vdash \PR_{T}^n (\gn{\sigma}) \to \sigma$.
By the law of excluded middle, $T \vdash \sigma$ holds.
By $n$-times applications of $\D{1}$, we obtain $T \vdash \PR_T^n(\gn{\sigma})$.
By the definition of $\sigma$, we have $T \vdash \neg \PR_{T}^n(\gn{\sigma})$, which is a contradiction.
\end{proof}

Proposition \ref{not arith1} implies the third clause of Theorem \ref{main}.
%Next, we verify that $\NA_{m,0}$ is not arithmetically sound with respect to $\Sigma_1$-sound theories.
Lastly, we prove the fourth clause of Theorem \ref{main}.
\begin{prop}\label{not arith2}
Let $T$ be $\Sigma_1$-sound. For each $m \geq 1$, there exists no $\Sigma_1$ provability predicate $\PR_T(x)$ of $T$ 
such that $p \to \Box^m p \in \PL (\PR_T)$.
\end{prop}
\begin{proof}
Suppose that there is a $\Sigma_1$ provability predicate $\PR_T(x)$ of $T$ such that $p \to \Box^m p \in \PL(\PR_T)$.
By the fixed-point lemma, we obtain a $\Pi_1$ sentence $\sigma$ such that 
 $\PA \vdash \sigma \leftrightarrow \neg \PR_T(\gn{\sigma})$, so $\sigma$ is a G\"{o}del sentence of $T$.
Let $f$ be an arithmetical interpretation such that $f(p) = \sigma$. 
We obtain $T \vdash \sigma \to \PR_{T}^m (\gn{\sigma})$,  
and $T \vdash \neg \PR_T(\gn{\sigma}) \to \PR_{T}^m (\gn{\sigma})$.
Since $T$ is $\Sigma_1$-sound,
we have $\mathbb{N} \models \neg \PR_T(\gn{\sigma}) \to \PR_{T}^m (\gn{\sigma})$.
We can easily prove that $T \nvdash \sigma$. Thus, we obtain $\mathbb{N} \not \models \PR_T(\gn{\sigma})$, and
$\mathbb{N} \models \PR_{T}^m (\gn{\sigma})$. Therefore,  it follows from the $\Sigma_1$-soundness of $T$ that $T \vdash \sigma$. This contradicts $T \nvdash \sigma$.
\end{proof}
Proposition \ref{not arith2} implies the fourth clause of Theorem \ref{main}.

Also, Proposition \ref{not arith2} is only applicable to $\Sigma_1$ provability predicates,
so we do not know whether $\NA_{m,0}$ becomes a provability logic of some $\PR_T(x)$ which is not $\Sigma_1$ if $T$ is $\Sigma_1$-sound.
We propose the following problem.
\begin{prob}\label{prob3}
Let $T$ be $\Sigma_1$-sound and $m \geq 1$.
Is there a $\Sigma_2$ provability predicate $\PR_T(x)$ of $T$ such that $\NA_{m,0} = \PL(\PR_T)$? 
\end{prob}
% Kurahashi proved that if $\PR_T(x)$ satisfies $\D{2}$ and $p \to \Box^m p \in \PL(\PR_T)$, then $\Box^m \bot \in \PL(\PR_T)$ (cf. \cite[Lemma 3.12]{Kur18}).
% Hence, $\PR_T(x)$ such that  $\NA_{m,0}= \PL(\PR_T)$ does not satisfy $\D{2}$ because $\NA_{m,0} \nvdash \Box^m \bot$.
 
\section{Proof of Theorem \ref{thm4-2}}\label{sec4}
First, we introduce some notions, which are used  throughout the rest of this paper. 
We call an $\LA$-formula  \textit{propositionally atomic} if it is either atomic or of the form $Q x \psi$, where $Q \in \{\forall, \exists  \}$.
For each propositionally atomic formula $\varphi$, we prepare a propositional variable $p_{\varphi}$.
Let $I$ be a primitive recursive injection from $\LA$-formulas into propositional formulas, which is defined as follows:
\begin{itemize}
\item 
$I(\varphi)$ is $p_{\varphi}$ for each propositionally atomic formula $\varphi$,
\item 
$I(\neg \varphi)$ is $\neg I(\varphi)$,
\item 
$I(\varphi \circ \psi)$ is $I(\varphi) \circ I(\psi)$ for $\circ \in \{ \wedge, \vee, \to \}$.
\end{itemize}
Let $X$ be a finite set of formulas.
An $\LA$-formula $\varphi$ is called a \textit{tautological consequence} (\textit{t.c.}) of $X$
if $\bigwedge_{\psi \in X}I(\psi) \to I(\varphi)$ is a tautology.  
Let $X \tc \varphi$ denote that $\varphi$ is a t.c.~of $X$.
Since t.c.~is a relation concerning propositional logic, for example, $X \tc \varphi \land \psi$ is equivalent to
$X \tc \varphi$ and $X \tc \psi$.
On the other hand, principles of predicate logic do not generally hold, 
for example, for a term $t$, $X \tc \forall x\, \varphi(x) \to \varphi(t)$
does not necessarily hold.
For each $n \in \omega$, let $P_{T,n} : = \{ \varphi \mid \mathbb{N} \models \exists y \leq \num{n} \ \Proof_T(\gn{\varphi}, y)  \}$. 
We see that  if 
$P_{T,n} \tc \varphi$, then $\varphi$ is provable in $T$.
We can formalize the above notions in the language of $\PA$.

Next, we prepare a $\PA$-provably recursive function $h$, which is originally introduced in \cite{Kur20}.
The function $h$ is defined by the recursion theorem as follows:
\begin{itemize}
	\item $h(0) = 0$. 
	\item $h(s+1) = \begin{cases} \text{min} \ J_s & \text{if}\ h(s) = 0 \ \text{and} \ J_s \neq \emptyset \\
				
			h(s) & \text{otherwise},
		\end{cases}$
\end{itemize}
where $J_s = \{ j \in \omega \setminus \{0\} \mid P_{T,s} \tc \neg S(\num{j}) \}$ and $S(x)$ is the $\Sigma_1$ formula $\exists y(h(y) = x)$. 
The function $h$ satisfies the property that for each $s$ and $i$, $h(s+1)=i$ implies $i \leq s+1$, which guarantees that the function $h$ is $\PA$-provably recursive (See \cite{Kur20}, p. 603).
The following proposition holds for $h$. 
\begin{prop}[Cf.~{\cite[Lemma 3.2]{Kur20}}]\label{Prop:h}
\leavevmode
\begin{enumerate}
	\item $\PA \vdash \forall x \forall y(0 < x < y \land S(x) \to \neg S(y))$. 
	\item $\PA \vdash \neg \Con_T \leftrightarrow \exists x(S(x) \land x \neq 0)$. 
	\item For each $i \in \omega \setminus \{0\}$, $T \nvdash \neg S(\num{i})$. 
	\item For each $l \in \omega$, $\PA \vdash \forall x \forall y(h(x) = 0 \land h(x+1) = y \land y \neq 0 \to x \geq \num{l})$. 
\end{enumerate}
\end{prop}
We are ready to prove Theorem \ref{thm4-2}.
\subsection{The case $n > m \geq 1$}\label{4.1}
In this subsection, we prove the following theorem, which is Theorem \ref{thm4-2} with respect to the case $n > m \geq 1$.
%the arithmetical completeness theorems of $\NA_{m,n}$ with respect to $n > m \geq 1$.
%We prove the following uniform version of the arithmetical completeness theorem.
\begin{thm}\label{thm4-4}
Suppose that $T$ is $\Sigma_1$-sound and $n > m \geq 1$. Then, there exists a $\Sigma_1$ provability
	predicate $\PR_T(x)$ of $T$ such that
	\begin{enumerate} 
	\item 
	for any $A \in \mathsf{MF}$ and arithmetical interpretation $f$ based on $\PR_T(x)$,
	if $\mathbf{N}\mathbf{A}_{m,n} \vdash A$, then $T \vdash f(A)$, and 
	\item 
	there exists an arithmetical interpretation $f$ based on $\PR_T(x)$ such that 
	$\mathbf{N}\mathbf{A}_{m,n} \vdash A$ if and only if $T \vdash f(A)$.
	\end{enumerate}
\end{thm}
\begin{proof}
Let $T$ be $\Sigma_1$-sound and $n > m \geq 1$.
First, we modify the standard provability predicate $\Prov_T(x)$ to make $\Sigma_1$-reflection principle for it provable in $\PA$. Here, $\Sigma_1$-reflection principle for $\PR_T(x)$ is the set of sentences $\{\PR_T(\gn{\sigma}) \to \sigma \mid \sigma \text{ is } \Sigma_1 \text{ sentence} \}$. 
Let $\Prov_{T}^{\Sigma}(x)$ be the $\Sigma_1$ formula $\Prov_T(x) \wedge \bigr( \Sigma_1(x) \to \True_{\Sigma_1}(x)\bigl)$. Here,
$\Sigma_1(x)$ is a $\Delta_1 (\PA)$ formula and $\True_{\Sigma_1}(x)$ is a $\Sigma_1$ formula, and they are introduced in Section \ref{sec2}.
We show that the $\Sigma_1$ formula $\Prov_{T}^{\Sigma}(x)$ is a $\Sigma_1$ provability predicate of $T$ proving $\Sigma_1$-reflection principle for $\Prov_{T}^{\Sigma}(x)$.

	\begin{cl}\label{Cl4-1}
For any formula $\varphi$, $T \vdash \varphi$ if and only if $\PA \vdash \Prov_T^{\Sigma}(\gn{\varphi})$.
\end{cl}
\begin{proof}
$( \Leftarrow)$: Suppose $\PA \vdash \Prov_{T}^{\Sigma}(\gn{\varphi})$. Then, $\PA \vdash \Prov_T(\gn{\varphi})$, and we get $T \vdash \varphi$.
$(\Rightarrow)$: Assume $T \vdash \varphi$. Then, $\PA \vdash \Prov_T(\gn{\varphi})$ holds. If $\varphi$ is a $\Sigma_1$ sentence, $T \vdash \varphi$ implies $\mathbb{N} \models \varphi$ by the $\Sigma_1$ soundness of $T$, 
and we get $\PA \vdash \varphi$. Then, we obtain $\PA \vdash \True_{\Sigma_1}(\gn{\varphi})$.
If $\varphi$ is not a $\Sigma_1$ sentence, then $\PA \vdash \neg \Sigma_1(\gn{\varphi})$ holds. In either case, it follows that $\PA \vdash \Prov_{T}^{\Sigma}(\gn{\varphi})$.
\end{proof}

\begin{cl}\label{Cl4-2}
For any $\Sigma_1$ sentence $\sigma$, $\PA \vdash \Prov_T^{\Sigma}(\gn{\sigma}) \to \sigma$.
\end{cl}
\begin{proof}
Let $\sigma$ be a $\Sigma_1$ sentence. Since $\PA \vdash \Sigma_1 (\gn{\sigma})$,
we obtain $\PA \vdash \Prov_{T}^{\Sigma}(\gn{\sigma}) \to \True_{\Sigma_1}(\gn{\sigma})$. 
It follows that $\PA \vdash \Prov_{T}^{\Sigma}(\gn{\sigma}) \to \sigma$.
\end{proof}

Let $\langle A_k  \rangle_{k \in \omega}$ denote a primitive recursive enumeration of all $\NA_{m,n}$-unprovable formulas.
For each $k \in \omega$,  
we can primitive recursively construct a finite $(m,n)$-accessible $\N$-model $\bigl( W_k, \{ \prec_{k,B}\}_{B \in \MF}, \Vdash_k   \bigr)$ which falsifies $A_k$ (see Fact \ref{fact2-7} and Fact \ref{fact2-8}).
We may assume that the sets $\{ W_k \}_{k \in \omega}$ are pairwise disjoint subsets of $\omega$ and $\bigcup_{k \in \omega} W_k = \omega \setminus \{0\}$. 
We may also assume that for each $i \in \omega \setminus \{0\}$, we can primitive recursively find a unique $k \in \omega$ satisfying $i \in W_k$.
We define a $(m,n)$-accessible $\N$-model $\mathcal{M}= \bigl( W, \{ \prec_{B}\}_{B \in \MF}, \Vdash \bigr)$ as follows:
\begin{itemize}
\item 
$W$ is $\bigcup_{k \in \omega} W_k = \omega \setminus \{0\}$,
\item 
$x \prec_{B} y$ if and only if $x, y \in W_k$ and $x \prec_{k,B} y$ for some $k \in \omega$,
\item
$x \Vdash p$ if and only if $x \in W_k$ and $x \Vdash_k p$ for some $k \in \omega$.
\end{itemize}
We may assume that each $\bigl( W_k, \{ \prec_{k,B}\}_{B \in \MF}, \Vdash_k   \bigr)$ is $\Delta_1(\PA)$ represented in $\PA$ and we see that 
basic properties of $\mathcal{M}$ are proved in $\PA$.

Next, we define a $\PA$-provably recursive function $g_0$ outputting all $T$-provable formulas step by step.
The definition of $g_0$ consists of Procedure 1 and Procedure 2, and starts with Procedure 1.
In Procedure 1, we define the values $g_0(0), g_0(1), \ldots$ 
 by referring to the values $h(0), h(1), \ldots$ and $T$-proofs based on $\Proof_T(x,y)$. 
At the first time $h(s+1) \neq 0$, the definition of $g_0$ switches to Procedure 2 at Stage $s$.

We define the $\PA$-provably recursive function $g_0$ by using the recursion theorem.  
In the definition of $g_0$, we use
the $\Sigma_1$ formula $\PR_{g_0}(x) \equiv$ $\exists y (g_0(y)=x \wedge \Fml(x))$
and the arithmetical interpretation $f_0$ based on $\PR_{g_0}(x)$ such that 
$f_0(p) \equiv \exists x (S(x) \wedge x \neq 0 \wedge x \Vdash p)$. 
Here, for each $A \in \MF$, $f_0(A)$ is $\PA$-provably recursively computed from $A$. Also, $A$ is $\PA$-provably recursively computed from $f_0(A)$
because $f_0$ is injection.
For a $\Sigma_1$ formula $\exists x \delta(x)$, where $\delta(x)$ is a $\Delta_0$ formula, and a natural number $s$, we say that $s$ is a witness of $\exists x \delta(x)$ if $\delta(\num{s})$ is a true $\Delta_0$ sentence. 
The definition of $g_0$ is as follows:
\medskip

\textbf{Procedure 1.}

Stage s:
\begin{itemize}
\item
If $h(s+1)=0$, then 
\begin{equation*}
	g_0(s)= \begin{cases} \varphi & \text{if}\  s \ \text{is a proof of} \ \varphi \ \text{and} \ \varphi \ \text{is ready at Stage} \ s \\
	0 & \text{otherwise},
			\end{cases}
	\end{equation*}
where a formula $\varphi$ \textit{is ready at Stage} $s$ if and only if one of the following conditions holds:
\paragraph{Case A:}
$\varphi$ is not a $\Sigma_1$ sentence;
\paragraph{Case B:}
$\varphi$ is a $\Sigma_1$ sentence 
which is not of the form $\PR_{g_0}(\gn{\psi})$
and there exists  $l <s$ which is a witness of $\True_{\Sigma_1}(\gn{\varphi})$;

\paragraph{Case C:}
$\varphi \equiv \PR_{g_0}(\gn{\psi})$ for some $\psi$ and there is $l<s$ such that $g_0(l)= \psi$.

%\paragraph{Case A:} 
%$s$ is a $T$-proof of a formula $\varphi$.
%\paragraph{Case A-1:} 
%If $\varphi$ is not a $\Sigma_1$ sentence, then $g_0(s) = \varphi$. 
%\paragraph{Case A-2:} 
%$\varphi$ is a $\Sigma_1$ sentence and $\varphi$ is not of the form $\PR_{g_0}(\gn{\psi})$ for any $\psi$.
%\paragraph{Case A-2-1:} 
%If there exists an $l <s$ such that $l$ is a witness of $\True_{\Sigma_1}(\gn{\varphi})$, then $g_0(s)= \varphi$.

%\paragraph{Case A-2-2:} 
%Otherwise. $g_0(s)=0$.
%\paragraph{Case A-3:} 
%$\varphi$ is a $\Sigma_1$ sentence and $\varphi$ is of the form $\PR_{g_0}(\gn{\psi})$ for some $\psi$.

%We can find a formula $\rho$ and $r \geq 1$ such that $\varphi \equiv \PR_{g_0}^r(\gn{\rho})$,
%where $\rho$ is not of the form of $\PR_{g_0}(\gn{\psi})$ for any $\psi$.
%Here, $r$ denotes the maximum number of outermost consecutive applications of $\PR_{g_0}(x)$ in $\varphi$.

%\paragraph{Case A-3-1:}
%If for each $0 \leq i \leq r-1$, there exists an $l_i < s$ such that $g_0(l_i) = \PR_{g_0}^i(\gn{\rho})$, then $g_0(s)= \varphi$.

%\paragraph{Case A-3-2:}
%Otherwise. Let $g_0(s)= 0$.

%\paragraph{Case B:}
%$s$ is not a $T$-proof of any formula $\varphi$. Let $g_0(s)=0$.

%Then, go to Stage $s+1$.

\item 
If $h(s+1) \neq 0$, then go to Procedure 2.
\end{itemize}

\textbf{Procedure 2.}
Suppose $h(s)=0$ and $h(s+1)=i \neq 0$. Let $k$ be a number such that $i \in W_k$.
Let $\{ \xi_t \}_{t \in \omega}$ denote the primitive 
recursive enumeration of all $\LA$-formulas introduced in Section \ref{sec2}. 
Define 
\begin{equation*}
	g_0(s+t)= \begin{cases} \xi_t & \text{if}\  \xi_t \equiv f_0(B)\  \&\ i \Vdash_k \Box B \ \text{for some}\ \Box B \in \mathsf{Sub}(A_k), \\
			                    
		0 & \text{otherwise}.
			\end{cases}
	\end{equation*}

We have completed the definition of $g_0$.
Here, we comment our construction of $g_0$.
In Procedure 1, we construct $g_0$ so that the behavior of $\PR_{g_0}(x)$ becomes the same as that of $\Prov_{T}^{\Sigma}(x)$.
In Procedure 2, $g_0$ is designed to associate $\Box B$ with $\PR_{g_0}(\gn{f_0(B)})$ by outputting $f_0(B)$ such that $i \Vdash \Box B$.
% Since $\Prov_T^{\Sigma}(x)$ proves $\Sigma_1$-reflection principle, and since countermodels of $\NA_{m,n}$ validate $\Box^n A \to \Box^m A$, these parts of the construction of $g_0$ make it possible for $\PR_{g_0}(x)$ to satisfy $\PA \vdash \PR_{g_0}^n(\gn{\varphi}) \to \PR_{g_0}^m(\gn{\varphi})$.
\medskip

The following claim ensures that $\PR_{g_0}(x)$ becomes a provability predicate of $T$.
\begin{cl}\label{Cl4-3}
For any formula $\varphi$, $\PA + \Con_T \vdash \PR_{g_0}(\gn{\varphi}) \leftrightarrow \Prov_T^\Sigma (\gn{\varphi})$.

\end{cl}
\begin{proof}
	We distinguish the following two cases.
	\paragraph{Case 1:} $\varphi$ is not a $\Sigma_1$ sentence.
	
	Since $\PA \vdash \neg \Sigma_1(\gn{\varphi})$, we obtain 
	\begin{equation}\label{prg}
	\PA \vdash \Prov_{T}^{\Sigma}(\gn{\varphi}) \leftrightarrow \Prov_T(\gn{\varphi}).
	\end{equation}
	We argue in $\PA + \Con_T$: By Proposition \ref{Prop:h}, the definition of $g_0$ never goes to Procedure 2.
By the definition of Procedure 1 of $g_0$, it follows that $\PR_{g_0}(\gn{\varphi})$ if and only if $\Prov_T(\gn{\varphi})$.
Then, we conclude that $\PR_{g_0}(\gn{\varphi})$ if and only if $\Prov_{T}^{\Sigma}(\gn{\varphi})$ by (\ref{prg}).

	\paragraph{Case 2:} $\varphi$ is a $\Sigma_1$ sentence. 

	We argue in $\PA + \Con_T$: By Proposition \ref{Prop:h}, the construction of $g_0$ never switches to Procedure 2. 
	
$(\to)$:
Suppose $\PR_{g_0}(\gn{\varphi})$. Then, $g_0(s) = \varphi$ for some $s$.
By the definition of $g_0$ at Stage $s$ in Procedure 1, $\Prov_T(\gn{\varphi})$ holds.
Since $\varphi$ is $\Sigma_1$,
the formula $\varphi$ is output in Case B and there exists $l < s$ such that $l$ is a witness of $\True_{\Sigma_1}(\gn{\varphi})$.
Then, $\True_{\Sigma_1}(\gn{\varphi})$ holds and we have proved that $\Prov_{T}^{\Sigma}(\gn{\varphi})$ holds.

% or Case C. 
% If $\varphi$ is not of the form $\PR_{g_0}(\gn{\psi})$,
% then $\varphi$ is output in Case B and there exists $l < s$ such that $l$ is a witness of $\True_{\Sigma_1}(\gn{\varphi})$.
% If $\varphi$ is of the form $\PR_{g_0}(\gn{\psi})$, 
% then $\varphi$ is output in Case C and there exists $l<s$ such that $g_0(l) = \psi$, that is, $\PR_{g_0}(\gn{\psi})$ is true.
% In both cases, $\True_{\Sigma_1}(\gn{\varphi})$ holds.
% Hence, we have proved that $\Prov_{T}^{\Sigma}(\gn{\varphi})$ holds.
    
%We consider the following two cases.
% \begin{itemize}
% \item 
% $\varphi$ is output in Case A-2-1:

% There exists an $l < s$ such that $l$ is a witness of $\True_{\Sigma_1}(\gn{\varphi})$.
% Hence, we have proved that $\Prov_{T}^{\Sigma}(\gn{\varphi})$ holds.
% \item    
% $\varphi$ is output in Case A-3-1:
% Let $\rho$ and $r \geq 1$ be such that $\varphi \equiv \PR_{g_0}^r(\gn{\rho})$, where $\rho$ is not of the form of $\PR_{g_0}(\gn{\psi})$ for any $\psi$.
% It follows that for each $i \leq r-1$, there exists an $l_i < s$ such that $g_0(l_i) = \PR_{g_0}^i (\gn{\rho})$.
% In particular, since $g_0(l_{r-1}) = \PR_{g_0}^{r-1}(\gn{\rho})$, it follows that $\PR_{g_0}^r(\gn{\rho})$ holds. Hence, $\True_{\Sigma_1}(\gn{\PR_{g_0}^r(\gn{\rho})})$ holds.
% Therefore,  we obtain $\Prov_{T}^{\Sigma}(\gn{\varphi})$.
% \end{itemize}

$(\leftarrow):$ Suppose $\Prov_{T}^{\Sigma}(\gn{\varphi})$. Then, we obtain $\Prov_T(\gn{\varphi})$ 
and $\True_{\Sigma_1}(\gn{\varphi})$.
It suffices to prove that $\varphi$ is ready at some Stage.
We consider the following two cases.
\begin{itemize}
\item 
$\varphi$ is not of the form $\PR_{g_0}(\gn{\psi})$:

Let $l$ be a witness of $\True_{\Sigma_1}(\gn{\varphi})$.
Since we have $\Prov_T(\gn{\varphi})$, there exists some $s > l$ such that $s$ is a $T$-proof of $\varphi$.
Thus, $\varphi$ is ready at Stage $s$.

\item 
$\varphi$ is of the form $\PR_{g_0}(\gn{\psi})$ for some $\psi$:

Since $\varphi \equiv \PR_{g_0}(\gn{\psi})$ is true, there is $l$ such that $g_0(l)= \psi$.
Since $\Prov_T(\gn{\varphi})$ holds,
there exists some $s>l$ such that
$\varphi$ is ready at Stage $s$.
\qedhere    
\end{itemize}
\end{proof}
By Claim \ref{Cl4-3}, for any $\varphi$, we obtain $\mathbb{N} \models \PR_{g_0}(\gn{\varphi}) \leftrightarrow \Prov_{T}^{\Sigma}(\gn{\varphi})$.
Since $\Prov_{T}^{\Sigma}(x)$ is a $\Sigma_1$ provability predicate of $T$ by Claim \ref{Cl4-1}, so is $\PR_{g_0}(x)$.
\begin{cl}\label{Cl4-4}
For any formula $\varphi$, $\PA \vdash \PR_{g_0}^n (\gn{\varphi}) \to \PR_{g_0}^m(\gn{\varphi})$.
\end{cl}
\begin{proof}
For each $k \geq 1$, we obtain $\PA \vdash \Prov_{T}^{\Sigma}(\gn{\PR_{g_0}^k(\gn{\varphi})}) \to \PR_{g_0}^k(\gn{\varphi})$ by Claim \ref{Cl4-2}.
By Claim \ref{Cl4-3}, $\PA + \Con_T \vdash \PR_{g_0}^{k+1} (\gn{\varphi}) \to \PR_{g_0}^{k}(\gn{\varphi})$.
Since $n > m \geq 1$, it follows that $\PA + \Con_T \vdash \PR_{g_0}^n (\gn{\varphi}) \to \PR_{g_0}^m (\gn{\varphi})$.

Next, we show that $\PA + \neg \Con_T \vdash \PR_{g_0}^n (\gn{\varphi}) \to \PR_{g_0}^m (\gn{\varphi})$.
We argue in $\PA + \neg \Con_T$: 
By Proposition \ref{Prop:h}.2, there exists some $i \neq 0$ satisfying $S(\num{i})$. 
Let $s$ and $k$ be such that $h(s)=0$ and $h(s+1)=i \in W_k$.
Then, the definition of $g_0$ switches to Procedure 2 at Stage $s$.
Suppose that $\PR_{g_0}^n(\gn{\varphi})$ holds, namely, 
$\PR_{g_0}^{n-1} (\gn{\varphi})$ is output by $g_0$. 
We show that $\PR_{g_0}^{m}(\gn{\varphi})$ holds, that is, $\PR_{g_0}^{m-1}(\gn{\varphi})$ is output by $g_0$.
We distinguish the following two cases.
\begin{itemize}
\item $\PR_{g_0}^{n-1}(\gn{\varphi})$ is output in Procedure 1: 

By the induction on $0 \leq i \leq n-1$, we prove $\PR_{g_0}^{n-1-i}(\gn{\varphi})$ is output by $g_0$ in Procedure 1.
If $i=0$, then $\PR_{g_0}^{n-1}(\gn{\varphi})$ is output in Procedure 1. 
Suppose $i+1 \leq n-1$.
By the induction hypothesis, $\PR_{g_0}^{n-1-i}(\gn{\varphi})$ is output by $g_0$ in Procedure 1.
Let $l$ be such that $g_0(l)= \PR_{g_0}^{n-1-i}(\gn{\varphi})$.
Since $n>m \geq 1$, we obtain
 $n-1 \geq 1$.  Since $n-1 \geq 1$ and $n-1 >i$,
$\PR_{g_0}^{n-1-i}(\gn{\varphi})$ is of the form $\PR_{g_0}(\gn{\psi})$. Thus, $\PR_{g_0}^{n-1-i}(\gn{\varphi})$ is output at Case C,
which implies that there is $r<l$ such that  $g_0(r) = \PR_{g_0}^{n-1-(i+1)}(\gn{\varphi})$. Thus, $\PR_{g_0}^{n-1-(i+1)}(\gn{\varphi})$ is output by $g_0$ in Procedure 1.
Therefore, for each $0 \leq i \leq n-1$, $\PR_{g_0}^{n-1-i}(\gn{\varphi})$ is output by $g_0$ in Procedure 1.
In particular, $\PR_{g_0}^{m-1}(\gn{\varphi})$ is output by $g_0$ in Procedure 1.

% Since $n>m \geq 1$, $\PR_{g_0}^{n-1}(\gn{\varphi})$ is of the form $\PR_{g_0}(\gn{\psi})$. The formula $\PR_{g_0}^{n-1}(\gn{\varphi})$ is output in Case B. 
% Let $\rho$ and $r \geq 0$ be such that $\varphi \equiv \PR_{g_0}^r (\gn{\rho})$, where $\rho$ is not of the form of $\PR_{g_0}(\gn{\psi})$.
% By the induction on $0 \leq i \leq n-1$, we prove $\PR_{g_0}^{n-1-i}(\gn{\varphi})$ is output by $g_0$.

% Since $n > m \geq 1$, $\PR_{g_0}^{n-1}(\gn{\varphi})$ is of the form $\PR_{g_0}(\gn{\psi})$ for some $\psi$.

% and is output at Case A-3-1 in Procedure 1.
% Hence, for each $i \leq n-1$, there exists an $l_i$ such that $g_0(l_i) = \PR_{g_0}^i (\gn{\varphi})$.
% Since $0 \leq m-1 < n-1$ holds, we obtain $g_0(l_{m-1})= \PR_{g_0}^{m-1}(\gn{\varphi})$.

\item $\PR_{g_0}^{n-1}(\gn{\varphi})$ is output in Procedure 2: 

Let $\xi_u \equiv \PR_{g_0}^{m-1}(\gn{\varphi})$.
It follows that 
$\PR_{g_0}^{n-1}(\gn{\varphi}) \equiv f_0(B)$ and $i \Vdash_k \Box B$ for some $\Box B \in \Sub(A_k)$.
Then, by $\PR_{g_0}^{n-1}(\gn{\varphi}) \equiv f_0(B)$, we obtain some $C \in \Sub(A_k)$ such that 
$f_0(C) = \varphi$ and $B \equiv \Box^{n-1}C$.
Since $i \Vdash_k \Box B$ and $i \Vdash_k \Box^n C \to \Box^m C$,
we get  $i \Vdash_k \Box^m C$.
It follows from $m \leq n-1$ and $B \in\Sub(A_k)$ that $\Box^m C \in \Sub(A_k)$ holds. 
Also, we have $f_0(\Box^{m-1}C) \equiv \PR_{g_0}^{m-1}(\gn{f_0(C)}) \equiv \xi_u$.
We conclude $g_0(s+u) = \xi_u \equiv \PR_{g_0}^{m-1}(\gn{\varphi})$.
\end{itemize}
Therefore, we obtain $\PA + \neg \Con_T \vdash \PR_{g_0}^n (\gn{\varphi}) \to \PR_{g_0}^m(\gn{\varphi})$. By the law of excluded middle, we conclude that $\PA \vdash \PR_{g_0}^n(\gn{\varphi}) \to \PR_{g_0}^m(\gn{\varphi})$.
\end{proof}

\begin{cl}\label{Cl4-5}
Let $i \in W_k$ and $B \in \Sub(A_k)$.
\begin{enumerate}
	\item
	If $i \Vdash_k B$, then $\PA \vdash S(\num{i}) \to f_{0}(B)$.
	\item 
	If $i \nVdash_k B$, then $\PA \vdash S(\num{i}) \to \neg f_{0}(B)$.

\end{enumerate}
\end{cl}
\begin{proof}
	We prove Clauses 1 and 2 simultaneously by induction on the construction of $B \in \Sub(A_k)$.
	We prove the base case of the induction. When $B$ is $\bot$,  Clauses 1 and 2 trivially hold. Suppose that $B \equiv p$.
	\begin{enumerate}
	\item  
	If $i \Vdash_k p$, then $\PA \vdash S(\num{i}) \to \exists x \bigl(S(x) \wedge x\neq0 \wedge x \Vdash p \bigr)$ holds.
	Hence, we get $\PA \vdash S(\num{i}) \to f_0(p)$.
	\item 
	Suppose $i \nVdash_k p$. By Proposition \ref{Prop:h}.1, $\PA \vdash S(\num{i}) \to \forall x (S(x) \wedge x \neq 0 \to x= \num{i})$.
	Thus, we obtain $\PA \vdash S(\num{i}) \to \forall x (S(x) \wedge x \neq 0 \to x \nVdash p)$ and $\PA \vdash S(\num{i}) \to \neg f_0(p)$ is proved.
	\end{enumerate}
	Next, we prove the induction cases. Since the cases of $\neg, \ \wedge,\ \vee$ and $\to$ can be  easily proved, we only consider the case $B \equiv \Box C$.

	1.	Suppose $i \Vdash_k \Box C$. We argue in $\PA + S(\num{i})$: Let $s$ be such that $h(s)=0$ and $h(s+1)=i$.
		Let $\xi_t \equiv f_0(C)$. Since $\Box C \in \Sub(A_k)$ and $i \Vdash_k \Box C$ hold,
		we obtain $g_0(s+t)= \xi_t$. Thus, $\PR_{g_0}(\gn{f_0(C)})$ holds, that is, it follows that $f_0(\Box C)$.
		
		2. Suppose $i \nVdash_k \Box C$. Then, there exists $j \in W_k$ such that
		$i \prec_{k,C} j$ and $j \nVdash_k C$. By the induction hypothesis, we obtain $\PA \vdash S(\num{j}) \to \neg f_0(C)$.
		Let $p$ be a proof of $S(\num{j}) \to \neg f_0(C)$ in $T$.

    We argue in $\PA + S(\num{i})$: Let $s$ be such that $h(s)=0$ and $h(s+1)=i$.
Suppose, towards a contradiction, that $f_0(C)$ is output by $g_0$.
\begin{enumerate}
\item[(i)]
$f_0(C)$ is output in Procedure 1: 

There exists $l<s$ such that
$l$ is a proof of $f_0(C)$ in $T$. Hence, $f_0(C) \in P_{T, s-1}$ holds.
By Proposition \ref{Prop:h}.4, $p<s$ holds, and we obtain $S(\num{j}) \to \neg f_0(C) \in P_{T, s-1}$.
Thus, it follows that $\neg S(\num{j})$ is a t.c.~of $P_{T,s-1}$. We obtain $h(s) \neq 0$, which is a contradiction.

\item [(ii)]
$f_0(C)$ is output in Procedure 2:

Then, $f_0(C) \equiv f_0(C')\  \text{and}\ i \Vdash_k \Box C' \ \text{for some}\ \Box C' \in \mathsf{Sub}(A_k)$ holds.
Since $f_0(C) \equiv f_0(C')$ and $f_0$ is an injection, we obtain $C \equiv C'$. Hence, $i \Vdash_k \Box C$ holds. This is a contradiction.  

\end{enumerate}

We have shown that 
$\PA + S(\num{i})$ proves that 
$f_0(C)$ is never output by $g_0$.
Therefore, we obtain $\PA+ S(\num{i}) \vdash \neg f_0(\Box C)$.
\end{proof}
We complete our proof of Theorem \ref{thm4-4}. The implication $(\Rightarrow)$ is obvious from Claim \ref{Cl4-4}.
We prove the implication $(\Leftarrow)$. Suppose $\NA_{m,n} \nvdash A$. Then, $A \equiv A_k$ for some $k \in \omega$ and $i \nVdash_k A$ for some $i \in W_k$.
Hence, we obtain $\PA \vdash S(\num{i}) \to \neg f_0(A)$ by Claim \ref{Cl4-5}. Thus, it follows from Proposition \ref{Prop:h}.3 that $T \nvdash f_0(A)$.
\end{proof}
%\begin{cor}Suppose that $T$ is $\Sigma_1$-sound and $n > m \geq 1$.There exists a $\Sigma_1$ provability predicate $\PR_T(x)$ such that
%$\NA_{m,n} = \PL(\PR_T)$.\end{cor}

\subsection{The case $m > n \geq 1$}
In this subsection, we verify Theorem \ref{thm4-5}, which is Theorem \ref{thm4-2} and Theorem \ref{thm4-3} with respect to the case $m>n \geq 1$.
In the proof of Theorem \ref{thm4-5}, we make no assumption as to whether $T$ is $\Sigma_1$-sound or not.
However, 
the proof of Theorem \ref{thm4-5} is not applicable to the case $m > n=0$ when $T$ is $\Sigma_1$-ill, so we prove this case in Section \ref{sec5}.
\begin{thm}\label{thm4-5}
Suppose $m>n\geq 1$. Then, there exists a $\Sigma_1$ provability
predicate $\PR_T(x)$ of $T$ such that
\begin{enumerate} 
\item 
for any $A \in \mathsf{MF}$ and arithmetical interpretation $f$ based on $\PR_T(x)$,
if $\mathbf{N}\mathbf{A}_{m,n} \vdash A$, then $T \vdash f(A)$, and 
\item 
there exists an arithmetical interpretation $f$ based on $\PR_T(x)$ such that 
$\mathbf{N}\mathbf{A}_{m,n} \vdash A$ if and only if $T \vdash f(A)$.
\end{enumerate}
\end{thm}
%Next, we verify the following uniform version of the arithmetical completeness theorem of $\NA_{m,n}$ for $m > n \geq 1$.
%This proof is a generalization of the proof of Fact \ref{fact2-2} (See Theorem 4.1 of \cite{Kur23}).
%\begin{thm}\label{thm:4-4}
%Suppose $T$ is $\Sigma_1$-sound and $m>n\geq 1$. There exists a $\Sigma_1$ provability
%predicate $\PR_T(x)$ of $T$ such that
%\begin{enumerate} 
%\item 
%for any $A \in \mathsf{MF}$ and arithmetical interpretation $f$ based on $\PR_T(x)$,
%if $\mathbf{N}^+\mathbf{A}_{m,n} \vdash A$, then $T \vdash f(A)$, and 
%\item 
%there exists an arithmetical interpretation $f$ based on $\PR_T(x)$ such that 
%$\mathbf{N}^+\mathbf{A}_{m,n} \vdash A$ if and only if $T \vdash f(A)$.
%\end{enumerate}
%\end{thm}

\begin{proof}
Let $m > n \geq 1$.
Let $\langle A_k  \rangle_{k \in \omega}$ be a primitive recursive enumeration of all $\NA_{m,n}$-unprovable formulas.
For each $k \in \omega$, we can primitive recursively construct a finite $(m,n)$-accessible $\N$-model $\bigl( W_k, \{ \prec_{k,B}\}_{B \in \MF}, \Vdash_k   \bigr)$  falsifying $A_k$.
As in Subsection \ref{4.1},
let $\mathcal{M} = \bigl( W, \{ \prec_{B}\}_{B \in \MF}, \Vdash  \bigr)$ be the $(m,n)$-accessible $\N$-model defined as a disjoint union of these finite $(m,n)$-accessible $\N$-models.

By the recursion theorem,
we define a $\PA$-provably recursive function $g_1$ corresponding to the case $m> n\geq 1$.
In the definition of $g_1$, we use
the $\Sigma_1$ formula $\PR_{g_1}(x) \equiv$ $\exists y (g_1(y)=x \wedge \Fml(x))$
and the arithmetical interpretation $f_1$ based on $\PR_{g_1}(x)$ such that 
$f_1(p) \equiv \exists x (S(x) \wedge x \neq 0 \wedge x \Vdash p)$.
\medskip

\textbf{Procedure 1.}

Stage $s$:
\begin{itemize}
\item If $h(s+1) =0$,
\begin{equation*}
  g_1(s)  = \begin{cases}
       \varphi & \text{if}\ s\ \text{is a}\ T\text{-proof of}\ \varphi, \\
               0 & \text{otherwise}.
             \end{cases}
  \end{equation*}

Then, go to Stage $s+1$.

\item If $h(s+1) \neq 0$, go to Procedure 2.
\end{itemize}

\textbf{Procedure 2.}

Suppose $s$ and $i \neq 0$ satisfy $h(s)=0$ and $h(s+1)=i$. 
Let $k$ be such that $i \in W_k$. 
Define

\begin{equation*}
g_1(s+t)= \begin{cases} \xi_t & \text{if}\  \xi_t \equiv f_1(B)\  \&\ i \Vdash_k \Box B \ \text{for some}\ \Box B \in \mathsf{Sub}(A_k) \\
	& \quad \text{or}\ \xi_t \equiv \PR_{g_1}^{m-1}(\gn{\varphi})\
  \&\ g_1(l) = \PR_{g_1}^{n-1}(\gn{\varphi}) \\
  & \quad \quad \quad \text{for some}\ \varphi\ \text{and}\ l<s+t,	 \\		                    
  
  0 & \text{otherwise}.
		\end{cases}
\end{equation*}

We have finished the definition of $g_1$. 
In Procedure 2, to ensure the condition $\PA \vdash \PR_{g_1}^n(\gn{\varphi}) \to \PR_{g_1}^m(\gn{\varphi})$, $g_1$ outputs $\PR_{g_1}^{m-1}(\gn{\varphi})$ if $g_1$ already outputs $\PR_{g_1}^{n-1}(\gn{\varphi})$.

\medskip

\begin{cl} \label{Cl:4-6}
$\PA + \Con_T \vdash \forall x \forall y \Bigl(\Fml(x) \to \bigl(\Proof_T(x, y) \leftrightarrow g_1(y) = x \bigr)\Bigr)$.
\end{cl}
\begin{proof}
We argue in $\PA + \Con_T$: By Proposition \ref{Prop:h}.2, we have $h(s) =0$ for any number $s$, and the definition of $g_1$ continues to be in Procedure 1. Thus, 
 it follows that for any formula $\varphi$ and any number $p$, $g_1 (p) = \varphi$ if and only if $p$ is a $T$-proof of $\varphi$. 
\end{proof}
By Claim \ref{Cl:4-6}, for any formula $\varphi$, we obtain $\mathbb{N} \models \PR_{g_1}(\gn{\varphi}) \leftrightarrow \Prov_T(\gn{\varphi})$, 
and we see that $\PR_{g_1}(x)$ is a $\Sigma_1$ provability predicate of $T$.
\begin{cl} \label{Cl:4-7}
For any $\LA$-formula $\varphi$, 
$\PA \vdash \PR_{g_1}^n (\gn{\varphi}) \to \PR_{g_1}^m(\gn{\varphi})$.
\end{cl}
\begin{proof}
First, we show $\PA + \Con_T \vdash \PR_{g_1}^n (\gn{\varphi}) \to \PR_{g_1}^m (\gn{\varphi})$. 
For any formula $\psi$, $\PR_{g_1}(\gn{\psi})$ is a $\Sigma_1$ sentence.
Then, $\PA \vdash \PR_{g_1}(\gn{\psi}) \to \Prov_T(\gn{\PR_{g_1}(\gn{\psi})})$.
By Claim \ref{Cl:4-6}, $\PA + \Con_T \vdash \Prov_T(\gn{\psi}) \leftrightarrow \PR_{g_1}(\gn{\psi})$. 
Hence, we obtain $\PA + \Con_T \vdash \PR_{g_1}(\gn{\psi}) \to \PR_{g_1}(\gn{\PR_{g_1}(\gn{\psi})})$.
In particular, it follows from $m > n \geq 1$ that $\PA + \Con_T \vdash \PR_{g_1}^n(\gn{\varphi}) \to \PR_{g_1}^m(\gn{\varphi})$.

Next, we show that $\PA + \neg \Con_T \vdash \PR_{g_1}^n(\gn{\varphi}) \to \PR_{g_1}^m(\gn{\varphi})$.
We argue in $\PA + \neg \Con_T + \PR_{g_1}^n(\gn{\varphi})$: By Proposition \ref{Prop:h}, $h(s)=0$ and $h(s+1)=i$ hold 
for some $i \neq 0$ and number $s$, and the definition of $g_1$ switches to Procedure 2 at Stage $s$.
Let $\xi_v \equiv \PR_{g_1}^{m-1}(\gn{\varphi})$ and we prove that $\xi_v$ is output by $g_1$ in Procedure 2.
Since $\PR_{g_1}^n(\gn{\varphi})$ holds, $\PR_{g_1}^{n-1}(\gn{\varphi})$ is output by $g_1$.
We consider the following two cases.
\begin{itemize}
\item 
$\PR_{g_1}^{n-1}(\gn{\varphi})$ is output in Procedure 1: 

Then, $g_1(l) = \PR_{g_1}^{n-1}(\gn{\varphi})$ holds for some $l < s$. 
Since $l < s+v$, $g_1(s+v)$ = $\xi_v$ holds by the definition of $g_2$ in Procedure 2. 
\item 
$\PR_{g_1}^{n-1}(\gn{\varphi})$ is output in Procedure 2: 

Then, $g_1(s+u) = \PR_{g_1}^{n-1}(\gn{\varphi})$, where $\xi_u \equiv \PR_{g_1}^{n-1}(\gn{\varphi})$.
Since $m-1 > n-1$, the G\"{o}del number of $\PR_{g_1}^{m-1}(\gn{\varphi})$ is larger than that of $\PR_{g_1}^{n-1}(\gn{\varphi})$. Hence, $v > u$ holds by the definition of $\{ \xi_t \}_{t \in \omega}$. 
Then, we obtain $g_1(s+v) = \xi_v$.
\end{itemize}
Thus, it follows that $\PA + \neg \Con_T \vdash \PR_{g_1}^n(\gn{\varphi}) \to \PR_{g_1}^m(\gn{\varphi})$. By the law of excluded middle, we conclude that $\PA \vdash \PR_{g_1}^n (\gn{\varphi}) \to \PR_{g_1}^m(\gn{\varphi})$.
\end{proof}

\begin{cl} \label{Cl:4-8}
Let $i \in W_k$ and $B \in \Sub(A_k)$.
\begin{enumerate}
\item
If $i \Vdash_k B$, then $\PA \vdash S(\num{i}) \to f_{1}(B)$.
\item 
If $i \nVdash_k B$, then $\PA \vdash S(\num{i}) \to \neg f_{1}(B)$.
\end{enumerate}
\end{cl}
\begin{proof}
We prove Clauses 1 and 2 simultaneously by induction on the construction of $B \in \Sub(A_k)$.
We only give a proof of the case  $B \equiv \Box C$.
We can prove Clause 1 similarly as in the proof of Claim \ref{Cl4-5}. 
We prove Clause 2.

Suppose $i \nVdash_k \Box C$. Then, there exists  $j \in W_k$ such that
$i \prec_{k,C} j$ and $j \nVdash_k C$. By the induction hypothesis, we obtain $\PA \vdash S(\num{j}) \to \neg f_1(C)$.
Let $p$ be a proof of $S(\num{j}) \to \neg f_1(C)$ in $T$. 
Let   $m' \geq 0$ and a formula 
$D \in \Sub(C)$ be such that $C \equiv \Box ^{m'} D$, 
where $D$ is not of the form $\Box G$ for any $G$. 
Here, if $C$ is a boxed formula,
$m'$ is the maximum number of the outermost consecutive boxes of $C$, and otherwise $m'=0$. 
We show that $\PA + S(\num{i})$ proves that  $f_1(C)$ is never output by $g_1$.
We distinguish the following two cases.
\paragraph{Case 1:} $m' \geq m-1$.

It follows from $m'\geq m-1$ that we can find $E \in \Sub(C)$ such that $C \equiv \Box^{m-1} E$.
Since $i \nVdash_k \Box C$, we obtain $i \nVdash_k \Box^m E$. Thus, $i \nVdash_k \Box^n E$ holds by $i \Vdash_k \Box^n E \to \Box^m E$.
Since $n \leq m-1$, by the induction hypothesis
\begin{equation}\label{p_0}
\PA \vdash S(\num{i}) \to \neg f_1(\Box^n E).
\end{equation}
We argue in $\PA + S(\num{i})$: Let $s$ be such that $h(s)=0$ and $h(s+1)=i$.
Suppose, towards a contradiction, that $f_1(C)$ is output by $g_1$.
We consider the following three cases.
\begin{enumerate}
\item[(i)]$f_1(C)$ is output in Procedure 1: 

For some $l<s$, $l$ is a proof of $f_1(C)$ in $T$ and $f_1(C) \in P_{T, s-1}$ holds.
By Proposition \ref{Prop:h}.4, we obtain $p<s$. Thus, $S(\num{j}) \to \neg f_1(C) \in P_{T, s-1}$ holds, which implies that $\neg S(\num{j})$ is a t.c.~of $P_{T,s-1}$.
Hence, it follws that $h(s) \neq 0$. This is a contradiction.
\item [(ii)]
$f_1(C) \equiv f_1(C')\  \text{and}\ i \Vdash_k \Box C' \ \text{for some}\ \Box C' \in \mathsf{Sub}(A_k)$ holds:

Since $f_1$ is an injection, we obtain $C \equiv C'$. Then, $i \Vdash_k \Box C$. This contradicts $i \nVdash_k \Box C$.  
\item [(iii)]
$f_1(C) \equiv \PR_{g_1}^{m-1}(\gn{\varphi})$ and $g_1(l) = \PR_{g_1}^{n-1}(\gn{\varphi})$
for some $\varphi$ and $l$ : 
Since $C \equiv \Box^{m-1}E$, we obtain $f_1(C) \equiv \PR_{g_1}^{m-1}(\gn{f(E)})$, so $\PR_{g_1}^{m-1}(\gn{\varphi}) \equiv \PR_{g_1}^{m-1}(\gn{f_1(E)})$ holds.
Then, $f_1(E) \equiv \varphi$ and $f_1(\Box^{n-1}E) \equiv \PR_{g_1}^{n-1}(\gn{f_1(E)}) \equiv \PR_{g_1}^{n-1}(\gn{\varphi})$.
Since $g_1(l) = f_1(\Box^{n-1}E)$, it follows that $\PR_{g_1}(\gn{f_1(\Box^{n-1}E)})$ holds. Hence, we obtain $f_1(\Box^n E)$, which contradicts  (\ref{p_0}). 
\end{enumerate}

\paragraph{Case 2:} $m' < m-1$

We argue in $\PA + S(\num{i})$: Let $s$ be such that $h(s)=0$ and $h(s+1)=i$.
Suppose, towards a contradiction, that $f_1(C)$ is output by $g_1$.
If $f_1(C)$ is output at the cases corresponding to (i) and (ii) in the proof of Case 1, then we can obtain a contradiction similarly.
If $f_1(C) \equiv \PR_{g_1}^{m-1}(\gn{\varphi})$ and $g_1(l) = \PR_{g_1}^{n-1}(\gn{\varphi})$ for some $\varphi$ and $l$, then $C$ must be of the form $\Box^{m-1}F$ for some $F$. 
Since $m' < m-1$, this is a contradiction.
\end{proof}
We finish our proof of Theorem \ref{thm4-5}. 
By Claim \ref{Cl:4-7}, we obtain the first clause.
By Proposition \ref{Prop:h}.3 and Claim \ref{Cl:4-8}, the second clause holds.
\end{proof}

%In the proof of Theorem \ref{thm:4-4}, we do not use the assumption that $T$ is $\Sigma_1$-sound at all.Thus, the proof of Theorem \ref{thm:4-4} is applicable for $T$ whether $T$ is $\Sigma_1$-sound or not.Especially, we obtain a proof of Fact \ref{fact2-2}.\begin{cor}There exists a $\Sigma_1$ provability predicate $\PR_T(x)$ such that
%$\mathbf{N4} = \PL(\PR_T)$.\end{cor}

\section{Proof of Theorem \ref{thm4-3}}\label{sec5}
In Theorem \ref{thm4-5}, we have proved Theorem \ref{thm4-3} with respect to the case $m > n \geq 1$.
In this section, 
we prove Theorem \ref{thm4-3} for the remaining cases $n > m \geq 1$ and $m > n=0$.
Before proving Theorem \ref{thm4-3}, we introduce a $\Sigma_1$ provability predicate $\PR_{\sigma}(x)$ and a $\PA$-provably recursive function $h'$ instead of the function $h$ defined in Section \ref{sec4}.
Throughout this section, we assume that $T$ is a $\Sigma_1$-ill theory. 

Since $T$ is $\Sigma_1$-ill, it is well-known that we obtain a $\Delta_1(\PA)$ formula $\sigma(x)$ representing $T$ in $\PA$ such that $\mathbb{N} \models \forall x (\tau(x) \leftrightarrow \sigma(x))$ and  $T \vdash \neg \Con_{\sigma}$ (cf.~\cite[Theorem 8 (b)]{Lin}).
%Here, $\tau(x)$ is the $\Delta_1(\PA)$ formula representing $T$ in $\PA$ introduced in Section \ref{sec2}.
%Remark that 
%Let $\delta(x)$ be a $\Delta_0$ formula such that $T \vdash \exists y \delta(y)$ and $\mathbb{N} \not \models \exists y \delta(y)$.
%By Orey's theorem (cf.~\cite[Theorem 8 (b)]{Lin}), we obtain the $\Delta_1(\PA)$ formula $\sigma(x) \equiv \tau(x) \vee \exists y \leq x \delta(y)$ representing $T$ in $\PA$ such that $T \vdash \neg \Con_{\sigma}$.
%Here, $\tau(x)$ is the $\Delta_1(\PA)$ formula representing $T$ in $\PA$ introduced in Section \ref{sec2}.
Remark that 
%we obtain $\mathbb{N} \models \forall x (\tau(x) \leftrightarrow \sigma(x))$. 
we get $\mathbb{N} \models \forall x \bigl(\PR_T(x) \leftrightarrow \PR_{\sigma}(x)\bigr)$.
In particular, since $T$ is consistent, $\mathbb{N} \models \Con_{\sigma}$ holds.
The set of sentences defined by $\sigma(x)$ in $\mathbb{N}$ is exactly $T$. On the other hand, the behaviors of $\sigma(x)$ and $\tau(x)$ may be different in $\PA$,
so let $T_{\sigma}$ denote the set of sentences defined by $\sigma(x)$ in $\PA$.
For each $n \in \omega$, let $P_{\sigma,n} := \{ \varphi \mid \mathbb{N} \models \exists y \leq \num{n} \ \Prf_{\sigma}(\gn{\varphi}, y)\}$. 
In $\PA$, we can see that  if $\varphi$ is a t.c.~of $P_{\sigma,n}$, then $\varphi$ is provable from $T_{\sigma}$.

In our proofs presented in this section, 
we use a $\PA$-provably recursive function $h'$ instead of the function $h$ introduced  in Section \ref{sec4}.
Let $\bigl\{ \bigl( W_k, \{ \prec_{k,B} \}_{B \in \MF}, \Vdash_k \bigr) \bigr\}_{k \in \omega}$ be a primitive recursive enumeration of some collection of $\N$-models 
such that the sets $\{ W_k \}_{k \in \omega}$ are pairwise disjoint subsets of $\omega$ and $\bigcup_{k \in \omega} W_k = \omega \setminus \{0\}$.
Let $\{ A_k\}_{k \in \omega}$ be a primitive recursive enumeration of $\mathcal{L}_{\Box}$-formulas.
We construct the function $h'$ by using the enumerations $\bigl\{ \bigl( W_k, \{ \prec_{k,B} \}_{B \in \MF}, \Vdash_k \bigr) \bigr\}_{k \in \omega}$ and $\{ A_k\}_{k \in \omega}$. 
The function $h'$ is defined by using the recursion theorem as follows:

\begin{itemize}
	\item $h'(0) = 0$. 
	\item $h'(s+1) = \begin{cases} \text{min} \ J'_s & \text{if}\ h'(s) = 0 \ \text{and} \ J'_s \neq \emptyset \\

		h'(s) & \text{otherwise},
		\end{cases}$
\end{itemize}
where 
\begin{align*}
J'_s = \{ j \in &  \omega\setminus \{0\}  \mid  P_{\sigma,s} \tc \neg S'(\num{j}) \ \text{or}\ \\
& \exists k  \exists B[ j \in W_k \ \&\ B \in \Sub(A_k) \ \& \ P_{\sigma,s} \tc \forall x \varphi_B(x) \wedge \bigl (\varphi_B(\num{j}) \to \neg S'(\num{j}) \bigr)]\}.
\end{align*}
Here, $S'(x)$ and $\varphi_B(x)$ are formulas 
\[
\exists y\bigl( h'(y)=x \bigr) \ \text{and} \ \bigl( x \neq 0 \wedge \exists y (x \in W_y \wedge B \in \Sub(A_y) \wedge x \nVdash_y B  ) \bigr) \to \neg S'(x),
\]
respectively.
%\exists y\bigl( h'(y)=x \bigr) 
%\ \text{and} \ \varphi_B(x) \equiv \bigl( x \neq 0 \wedge \exists y (x \in W_y \wedge B \in \Sub(A_y) \wedge x \nVdash_y B  ) \bigr) \to \neg S'(x)
We have finished the definition of $h'$.
\medskip

$\PA$-provably recursiveness of the function $h'$ is guaranteed by the following claim.
\begin{cl}\label{cl:6-3}
For any $j \neq 0$, if $h'(s)=0$ and $h'(s+1)=j$, then $j \leq s+1$. 

\end{cl}
\begin{proof}
Suppose $h'(s)=0$ and $h'(s+1)=j \in W_k$ for some $k$. If $P_{\sigma, s}$ is not propositionally satisfiable,
then $\neg S'(\num{1})$ is a t.c.~of $P_{\sigma, s}$. Hence, $j=1 \leq s+1$ holds.

Then, we assume that $P_{\sigma, s}$ is propositionally satisfiable. 
We distinguish the following two cases.
\paragraph{Case 1}
$\neg S'(\num{j})$ is a t.c.~of $P_{\sigma, s}$:

Since $S'(\num{j})$ is propositionally atomic, $S'(\num{j})$ is a subformula of a formula contained in $P_{\sigma, s}$. Thus, we obtain $j \leq s$. 
\paragraph{Case 2}
$\varphi_B (\num{j}) \to \neg S'(\num{j})$ is a t.c.~of $P_{\sigma, s}$ for some $B \in \Sub(A_k)$:

Then, we further consider the following two cases.
\begin{itemize}
\item
$\neg S'(\num{j})$ is a t.c.~of $P_{\sigma,s}$:

Since $S'(\num{j})$ is a subformula of a formula contained in $P_{\sigma, s}$, we obtain $j \leq s$.
\item 
$\neg S'(\num{j})$ is not a t.c.~of $P_{\sigma,s}$:

Since $j=0 \to \varphi_B(\num{j})$ is a tautology, $j=0 \to \neg S'(\num{j})$ is a t.c.~of $P_{\sigma,s}$.
If $j=0$ were not a subformula of any formula contained in $P_{\sigma,s}$, then $\neg S'(\num{j})$ would be a t.c.~of $P_{\sigma,s}$
because $j=0$ is distinct from $S'(\num{j})$.
Thus, $j=0$ is a subformula of a formula which is in $P_{\sigma,s}$, and we obtain $j \leq s$.
\end{itemize}
In either case, $j \leq s+1$ holds.
\qedhere
\end{proof}

We obtain the following proposition corresponding to Proposition \ref{Prop:h}.
\begin{prop}\label{prop:h'}
	\leavevmode

\begin{enumerate}
\item 
$\PA \vdash \forall x \forall y \bigl( 0 < x < y \wedge S'(x) \to \neg S'(y) \bigr)$.
\item 
$\PA \vdash \neg \Con_{\sigma} \leftrightarrow \exists x \bigl(S'(x) \wedge x \neq 0 \bigr)$.
\item 
For each $i \in \omega \setminus \{ 0\}$, $T \nvdash \neg S'(\num{i})$.
\item 
For each $l \in \omega$, $\PA \vdash \forall x \forall y \bigl( h'(x) =0 \wedge h'(x+1)=y \wedge y \neq 0 \to x \geq l \bigr)$.
\end{enumerate}

\end{prop}
\begin{proof}
1. Immediate from the definition of $h'$.
 
2. We work in $\PA$:
$(\leftarrow)$: Let $i \neq 0$ be such that $S'(\num{i})$. Let $k$ and $s$ be such that $i \in W_k$, $h'(s)=0$ and $h'(s+1)=i$.
We distinguish the following two cases.
\begin{itemize}
\item 
$\neg S'(\num{i})$ is a t.c.~of $P_{\sigma, s}$: 

Then, $\neg S'(\num{i})$ is provable in $T_\sigma$.
Since $S'(\num{i})$ is a true $\Sigma_1$ sentence, $S'(\num{i})$ is provable in $T_\sigma$ by the $\Sigma_1$-completeness. It follows that $T_\sigma$ is inconsistent.
\item 
$\forall x \varphi_B(x)$ and $\varphi_B(\num{i}) \to \neg S'(\num{i})$ are  t.c.'s~of $P_{\sigma, s}$ for some $B \in \Sub(A_k)$: 
Then, $\forall x \varphi_B(x)$ and $\varphi_B(\num{i}) \to \neg S'(\num{i})$ are provable in $T_\sigma$. Hence, $\varphi_B(\num{i})$ is provable in $T_\sigma$,
and so is $\neg S'(\num{i})$. 
Since $S'(\num{i})$ is a true $\Sigma_1$ sentence, $S'(\num{i})$ is provable in $T_\sigma$. We conclude that $T_\sigma$ is inconsistent.
\end{itemize}
In all cases, we obtain that $T_\sigma$ is inconsistent.

$(\to)$: If $T_\sigma$ is inconsistent, then $\neg S'(\num{i})$ is $T_\sigma$-provable for some $i \neq 0$. Let $s$ be a proof of $\neg S'(\num{i})$.
Then, $\neg S'(\num{i})$ is a t.c.~of $P_{\sigma, s}$, and we obtain $h'(s+1) \neq 0$.

3. Suppose that there exists an $i \in \omega \setminus \{0\}$ such that $T \vdash \neg S'(\num{i})$. Let $p$ be a proof of $\neg S'(\num{i})$ in $T$.
Then, $\neg S'(\num{i})$ is a t.c.~of $P_{\sigma,p}$ and we obtain $\mathbb{N} \models \exists x \bigl( S'(x) \wedge x \neq 0 \bigr)$. By clause 2, we obtain $\mathbb{N} \models \neg \Con_{\sigma}$. 
This contradicts the consistency of $T$.

4. Since $\mathbb{N} \models \Con_{\sigma}$, for any $l \in \omega$, we obtain $\mathbb{N} \models h'(\num{l}) =0$. Thus, $\PA \vdash h'(\num{l})=0$,
and clause 4 is easily obtained.
\end{proof}
We are ready to prove Theorem \ref{thm4-3}.
\subsection{The case $n > m \geq 1$}
In this subsection, we verify the following theorem, which is Theorem \ref{thm4-3} with respect to the case $n>m \geq 1$.
\begin{thm}\label{thm4-6}
Suppose that $T$ is $\Sigma_1$-ill and $n > m \geq 1$. Then, there exists a $\Sigma_1$ provability
	predicate $\PR_T(x)$ of $T$ such that 
\begin{enumerate} 
\item 
for any $A \in \mathsf{MF}$ and arithmetical interpretation $f$ based on $\PR_T(x)$,
if $\mathbf{N}\mathbf{A}_{m,n} \vdash A$, then $T \vdash f(A)$, and 
\item 
there exists an arithmetical interpretation $f$ based on $\PR_T(x)$ such that 
$\mathbf{N}\mathbf{A}_{m,n} \vdash A$ if and only if $T \vdash f(A)$.
\end{enumerate}
\end{thm}
%predicate $\PR_T(x)$ of $T$ such that
%\begin{enumerate} 
%\item 
%for any $A \in \mathsf{MF}$ and arithmetical interpretation $f$ based on $\PR_T(x)$,
%if $\mathbf{N}^+\mathbf{A}_{m,n} \vdash A$, then $T \vdash f(A)$, and 
%\item 
%there exists an arithmetical interpretation $f$ based on $\PR_T(x)$ such that 
%$\mathbf{N}^+\mathbf{A}_{m,n} \vdash A$ if and only if $T \vdash f(A)$.
%\end{enumerate}
%We obtain the following uniform version of the arithmetical completeness theorem.
%\begin{thm}\label{thm6-1}
%Suppose that $T$ is not $\Sigma_1$-sound and $n>m\geq 1$. There exists a $\Sigma_1$ provability
%predicate $\PR_T(x)$ of $T$ such that
%\begin{enumerate} 
%\item 
%for any $A \in \mathsf{MF}$ and arithmetical interpretation $f$ based on $\PR_T(x)$,
%if $\mathbf{N}^+\mathbf{A}_{m,n} \vdash A$, then $T \vdash f(A)$, and 
%\item 
%there exists an arithmetical interpretation $f$ based on $\PR_T(x)$ such that 
%$\mathbf{N}^+\mathbf{A}_{m,n} \vdash A$ if and only if $T \vdash f(A)$.
%\end{enumerate}
	
%\end{thm}
\begin{proof}
Let $n>m \geq 1$.
Let $\langle A_k  \rangle_{k \in \omega}$ denote a primitive recursive enumeration of all $\NA_{m,n}$-unprovable formulas.
As in the proof of Theorem \ref{thm4-2}, for each $k \in \omega$, we obtain a primitive recursively constructed finite $(m,n)$-accessible $\N$-model $\bigl( W_k, \{ \prec_{k,B}\}_{B \in \MF}, \Vdash_k   \bigr)$  falsifying $A_k$ 
and  let $\mathcal{M} = \bigl( W, \{\prec_{B}\}_{B \in \MF}, \Vdash \bigr)$ be the $(m,n)$-accessible $\N$-model defined as a disjoint union of these models.
We prepare the $\PA$-provably recursive function $h'$ constructed from the enumerations $\langle A_k  \rangle_{k \in \omega}$ and $\bigl \{ \bigl( W_k, \{ \prec_{k,B}\}_{B \in \MF}, \Vdash_k   \bigr) \bigr \}_{k \in \omega}$.

By the recursion theorem,
we define a $\PA$-provably recursive function $g_2$ corresponding to Theorem \ref{thm4-6}.
In the definition of $g_2$, we use 
the $\Sigma_1$ formula $\PR_{g_2}(x) \equiv$ $\exists y (g_2(y)=x \wedge \Fml(x))$
and the arithmetical interpretation $f_2$ based on $\PR_{g_2}(x)$ such that 
$f_2(p) \equiv \exists x (S'(x) \wedge x \neq 0 \wedge x \Vdash p)$.
%For each $r \in \omega$ and $\varphi$, $\PR_{g_2}^r(\gn{\varphi})$ is defined in the same way as $\PR_{g_0}^r(\gn{\varphi})$.
% In Procedure 1, $g_2$ outputs all formulas which are provable from $T_{\sigma}$.
% Also, if $g_2$ does not output $\PR_{g_2}^{m-1}(\gn{f_2(B)})$ in Procedure 2, then $\PR_{g_2}^{n-1}(\gn{f_2(B)})$ is not output by $g_2$ in Procedure 2.

\medskip

\textbf{Procedure 1.}

Stage $s$:
\begin{itemize}
\item If $h'(s+1) =0$,
\begin{equation*}
  g_2(s)  = \begin{cases}
       \varphi & \text{if}\ s\ \text{is a}\ T_\sigma\text{-proof of}\ \varphi, \\
               0 & \text{otherwise}.
             \end{cases}
  \end{equation*}

Then, go to Stage $s+1$.

\item If $h'(s+1) \neq 0$, go to Procedure 2.
\end{itemize}

\textbf{Procedure 2.}

Suppose $s$ and $i \neq 0$ satisfy $h'(s)=0$ and $h'(s+1)=i$. 
Let $k$ be a number such that $i \in W_k$. 
Define 

\begin{equation*}
g_2(s+t)= \begin{cases} 0 & \text{if}\  \xi_t \equiv f_2(B)\  \&\ i \nVdash_k \Box B \ \text{for some}\ \Box B \in \mathsf{Sub}(A_k) \\
  &  \quad \text{or}\ \xi_t \equiv \PR_{g_2}^{n-1}(\gn{f_2(B)}) \ \&\   \ g_2(s+u)=0, \\		                    
  & \quad \quad \text{where} \ \xi_u \equiv \PR_{g_2}^{m-1}(\gn{f_2(B)}) \ \text{for some}\  B \in \MF \ \text{and}\ u<t, \\
  \xi_t & \text{otherwise}.
		\end{cases}
\end{equation*}
The construction of $g_2$ has been finished.
We briefly remark on the definition of $g_2$.
% In Procedure 1, $g_2$ outputs all formulas which are provable from $T_{\sigma}$.
The first condition of Procedure 2 is slightly different from the previous ones.
The function $g_2$ establishes the correspondence between $\neg \Box B$ and $\neg \PR_{g_2}(\gn{f_2(B)})$
by not outputting $f_2(B)$ such that $i \nVdash \Box B$.
Also, if $g_2$ does not output $\PR_{g_2}^{m-1}(\gn{f_2(B)})$ in Procedure 2, then $\PR_{g_2}^{n-1}(\gn{f_2(B)})$ is not output by $g_2$ in Procedure 2,
which ensures the condition $\PA \vdash \PR_{g_1}^n(\gn{\varphi}) \to \PR_{g_1}^m(\gn{\varphi})$.
\medskip

The following claim guarantees that $\PR_{g_2}(x)$ is a $\Sigma_1$ provability predicate of $T$.
\begin{cl}\label{cl:6-4}
For any formula $\varphi$,
$\PA + \Con_{\sigma} \vdash \PR_{g_2}(\gn{\varphi}) \leftrightarrow \PR_{\sigma}(\gn{\varphi})$.
\end{cl}

\begin{proof}
We work in $\PA + \Con_{\sigma}$: By Proposition \ref{prop:h'}.2, the definition of $g_2$ never switches to Procedure 2.
Hence, for any $s$, $g_2(s)= \varphi$ if and only if $s$ is a proof of $\varphi$ in $T_\sigma$. 
\end{proof}

\begin{cl}\label{cl:6-5}
For any $C \in \MF$,
\[
\PA \vdash \exists x \bigl( x \neq 0 \wedge S'(x) \wedge \exists y(x \in W_y \wedge x \Vdash_y \Box C) \bigr) \to f_2(\Box C).
\]
\end{cl}
\begin{proof}
Let $m' \geq 1$ and let $D$ be such that $\Box C \equiv \Box^{m'} D$,
where $D$ is not a boxed formula.
Here, $m' \geq 1$ denotes the maximum number of the outermost consecutive boxes of $\Box C$.
We distinguish the following two cases.
\paragraph{Case 1:} $m' < n$.

We work in $\PA$: Let $i \neq 0$ and $k$ be such that $S'(\num{i})$, $i \in W_k$, and $i \Vdash_k \Box C$. 
Let $s$ be such that $h'(s) =0$ and $h'(s+1) =i$. Then, the definition of $g_2$ goes to Procedure 2 at Stage $s$.
We show that $f_2(C)$ is output by $g_2$ in Procedure 2.
Suppose, towards a contradiction, that $f_2(C)$ is not output by $g_2$ in Procedure 2.
Let $\xi_t \equiv f_2(C)$. Then, $g_2(s+t) = 0$. We consider the following two cases.
\begin{itemize}
\item 
$\xi_t \equiv f_2(E)$ and $i \nVdash_k \Box E$ for some $\Box E \in \Sub(A_k)$: Then, we obtain $C \equiv E$.
Hence, $i \nVdash \Box C$ holds, which is a contradiction.
\item 
$\xi_t \equiv \PR_{g_2}^{n-1}(\gn{f_2(E)})$ and  $g_2(s+u)=0$, where $\xi_u \equiv \PR_{g_2}^{m-1}(\gn{f_2(E)})$ for some $E \in \MF$ and $u<t$:
Since $f_2(C) \equiv \PR_{g_2}^{n-1}(\gn{f_2(E)})$, we obtain $C \equiv \Box^{n-1}E$. Since $\Box C \equiv \Box^{n} E$, this contradicts  the assumption that $m' <n $.
\end{itemize}
Therefore, $f_2(C)$ is output by $g_2$ in Procedure 2.

\paragraph{Case 2:} $m' \geq n$.

Since $m'-m \geq n-m$, we can uniquely find numbers $r<n-m$ and $q \geq 1$ such that $m'-m = q(n-m) +r$.
Hence, $\Box^{m'} D \equiv \Box^{m} \Box^{q(n-m)}\Box^r D$. Let $\Box^{r}D \equiv E$. Then, it follows that $\Box^{m'} D \equiv \Box^{q(n-m)} \Box^{m} E$.

We work in $\PA$: Let $i \neq 0$ and $k$ be such that $S'(\num{i})$, $i \in W_k$, and $i \Vdash_k \Box C$. 
Let $s$ be such that $h'(s) =0$ and $h'(s+1) =i$. Then, the definition of $g_2$ switches to Procedure 2 on Stage $s$.
We prove the following subclaims.
\begin{scl}\label{subcl1}
For each $0 \leq j \leq q$, we have $i \Vdash_k \Box^{(q-j) (n-m)} \Box^{m} E$.
\end{scl}
\begin{proof}
	We prove the statement by  induction on $j \leq q$.
	If $j=0$, then we obtain $i \Vdash_k \Box^{q(n-m)}\Box^m E$ by $i \Vdash_k \Box C \equiv \Box^{m'} D$.
	We prove on $j +1 \leq q$.
	By the induction hypothesis, we obtain $i \Vdash_k \Box^{(q-j)(n-m)}\Box^m E$,
	and we get $i \Vdash_k \Box^{n} \Box^{(q-j-1)(n-m)} E$. 
	Then, it follows from $\mathrm{Acc}_{m,n}$ that $i \Vdash_k \Box^{m} \Box^{(q-(j+1))(n-m)} E$.
\end{proof}

\begin{scl}\label{subcl2}
\leavevmode
For each $0 \leq j \leq q$, $f_2(\Box^{j(n-m)}\Box^{m-1}E)$
is output by $g_2$ in Procedure 2.

%Let $\xi_{t_j}$ denote $f_2(\Box^{j(n-m)}\Box^{m-1}E)$.
%Then, for each $0 \leq j \leq q$, $g_2(s+t_j)= \xi_{t_j}$. 

\end{scl}
\begin{proof}
For each $0 \leq j \leq q$, let $t_j$ be such that $\xi_{t_j} \equiv f_2(\Box^{j(n-m)}\Box^{m-1}E)$.
We prove $g_2(s+t_j)= \xi_{t_j}$ by induction on $j \leq q$. We prove the base case $j=0$. Suppose, towards a contradiction, that $g_2(s+t_0)=0$.
Since $r < n-m$, we have $m-1+r< n-1$. Then, $\xi_{t_0} \equiv \PR_{g_2}^{m-1}(\gn{f_2(E)})$ is not of the form $\PR_{g_2}^{n-1}(\gn{f_2(G)})$ for any $G$.
Hence, by $g_2(s+t_0)=0$, it follows that 
$\xi_{t_0} \equiv f_2(F)$ and $i \nVdash_k \Box F$ for some $\Box F \in \Sub(A_k)$.
Since $f_2(\Box^{m-1} E) \equiv f_2(F)$, it follows that $\Box^{m-1} E \equiv F$. Hence, we obtain $i \nVdash_k \Box^m E$.
This contradicts Subclaim \ref{subcl1}. Thus, we get $g_2(s+t_0)= \xi_{t_0}$.

Next, we prove the induction case $j+1 \leq q$. 
By the induction hypothesis, we obtain $g_2(s+t_j)= \xi_{t_j}$. Suppose, towards a contradiction, that $g_2(s+t_{j+1})=0$.
We consider the following two cases.
\begin{itemize}
\item $\xi_{t_{j+1}} \equiv f_2(F)$ and $i \nVdash_k \Box F$ hold for some $\Box F \in \Sub(A_k)$: 
We get $\Box^{(j+1)(n-m)} \Box^{m-1}E \equiv F$. Thus, we obtain $i \nVdash_k \Box^{(j+1)(n-m)} \Box^{m}E$, which contradicts Subclaim \ref{subcl1}.

\item $\xi_{t_{j+1}} \equiv \PR_{g_2}^{n-1}(\gn{f_2(F)})$ and $g_2(s+u)=0$, where $\xi_u \equiv \PR_{g_2}^{m-1}(\gn{f_2(F)})$ for some $F \in \MF$ and $u< t_{j+1}$:

Then, we obtain $\Box^{n-1}F \equiv \Box^{(j+1)(n-m)} \Box^{m-1}E$. It follows that $\Box^{m-1}F \equiv \Box^{j(n-m)} \Box^{m-1}E$, so we obtain $\xi_u \equiv \xi_{t_j}$ and $u=t_j$.
Hence, we obtain $g_2(s+t_j) = g_2(s+u) =0$. This contradicts $g_2(s+t_j)=\xi_{t_j}$.
\end{itemize}
We have proved $g_2(s+t_{j+1})= \xi_{t_{j+1}}$.
\end{proof}
By Subclaim \ref{subcl2}, we obtain $f_2(C)$ is output by $g_2$, that is, $f_2(\Box C)$ holds.
\end{proof}

\begin{cl}\label{cl:6-6}
Let $B \in \MF$.
\begin{enumerate}
\item 
$\PA \vdash \exists x \bigl( x \neq 0 \wedge S'(x) \wedge \exists y(x \in W_y \wedge B \in \Sub(A_y) \wedge x \Vdash_y B)  \bigr) \to f_2(B)$.
\item 
$\PA \vdash \exists x \bigl( x \neq 0 \wedge S'(x) \wedge \exists y(x \in W_y \wedge B \in \Sub(A_y) \wedge x \nVdash_y B)  \bigr) \to \neg f_2(B)$.
\end{enumerate}

\end{cl}
\begin{proof}
We prove Clauses 1 and 2 simultaneously by induction on the construction of $B \in \Sub(A_k)$.
First, we prove the base step of the induction. In the case $B \equiv \bot$, Clauses 1 and 2 are easily proved.
Suppose $B \equiv p$. Clause 1 holds trivially. We only give a proof of Clause 2. 
\begin{enumerate}
\item[2.]
By Proposition \ref{prop:h'}.1, we obtain 
\begin{align*}
\PA \vdash x \neq 0 \wedge S'(x) \wedge x \in W_y \wedge x \nVdash_y p &\to \bigl( S'(z) \wedge z\neq 0 \to z=x \bigr) \\
                                                                       &\to \forall z \bigl( S'(z) \wedge z\neq 0 \to z \nVdash p \bigr).
\end{align*}
Then, we get 
\[
\PA \vdash \exists x \bigl( x \neq 0 \wedge S'(x) \wedge \exists y(x \in W_y \wedge p \in \Sub(A_y) \wedge x \nVdash_y p)  \bigr) \to \neg f_2(p).
\]
\end{enumerate}
Next, we prove the induction step. The cases of $\neg$, $\wedge$, $\vee$, and $\to$ are easily proved.
We prove the case $B \equiv \Box C$.

1. We obtain Clause 1 by Claim \ref{cl:6-5}.
 
2. If $\mathbf{N} \mathbf{A}_{m,n} \vdash \Box C$, then $\PA \vdash \forall x \forall y(x \in W_y \to x \Vdash_y \Box C)$ holds by formalizing the soundness of $\NA_{m,n}$.
Hence, Clause 2 trivially holds.
Therefore, we may assume that $\mathbf{N} \mathbf{A}_{m,n} \nvdash \Box C$.
Then, there exists  $ j \in \omega$ such that $\Box C \equiv A_j$. Let $\bigl( W_j, \{ \prec_{j,D} \}_{D \in \MF}, \Vdash_j  \bigr)$ be a countermodel of $A_j$
and $r \in W_j$ be such that $r \nVdash_j \Box C$.
Then, there exists  $l \in W_j$ such that $r \prec_{j,C} l$ and $l \nVdash_j C$.
Thus, we obtain 
\[
\PA \vdash \num{l} \neq 0 \wedge \exists y(\num{l} \in W_y \wedge C \in \Sub(A_y) \wedge \num{l} \nVdash_y C), 
\]
which implies that $\PA \vdash \varphi_C(\num{l}) \to \neg S'(\num{l})$. 
Also, since 
\[
\PA \vdash \exists x \neg \varphi_C(x) \leftrightarrow \exists x \bigl( x \neq 0 \wedge S'(x) \wedge \exists y(x \in W_y \wedge C \in \Sub(A_y) \wedge x \nVdash_y C)  \bigr),
\]
it follows from the induction hypothesis that $\PA \vdash \exists x \neg \varphi_C(x) \to \neg f_2(C)$.
Let $p$ and $q$ be $T$-proofs of $\varphi_C(\num{l}) \to \neg S'(\num{l})$ and $\exists x \neg \varphi_C(x) \to \neg f_2(C)$ respectively.

We work in $\PA$: Let $i \neq 0$ and $k$ be such that
$S'(\num{i})$, $i \in W_k$, $\Box C \in \Sub(A_k)$ and $i \nVdash_k \Box C$.
Let $s$ be such that $h'(s)=0$ and $h'(s+1)=i$. Suppose, towards a contradiction, that $f_2(C)$ is output by $g_2$.
We consider the following two cases.
\begin{itemize}
\item 
$f_2(C)$ is output in Procedure 1:

Then, we obtain $f_2(C) \in P_{\sigma, s-1}$ and it follows from Proposition \ref{prop:h'} that $q \leq s-1$. Hence,  $\exists x \neg \varphi_C(x) \to \neg f_2(C)\in P_{\sigma,s-1}$ holds, which implies that
$\forall x \varphi_C(x)$ is a t.c.~of $P_{\sigma, s-1}$.
Since $p \leq s-1$ by Proposition \ref{prop:h'}, $\varphi_C(\num{l}) \to \neg S'(\num{l})$ is a t.c.~of $P_{\sigma, s-1}$.
We also have $l \in W_j$ and $C \in \Sub(A_j)$.
Hence, it follows  that $h'(s) \neq 0$,
 which is a contradiction.

\item 
$f_2(C)$ is output in Procedure 2: 

Let $\xi_u \equiv f_2(C)$.
Since $\Box C \in \Sub(A_k)$, $\xi_u \equiv f_2(C)$ and $i \nVdash_k \Box C$, we obtain $g_2(s+u)$=0, which means that $f_2(C)$ is never output in Procedure 2. This is a contradiction.
\end{itemize}
Therefore, $f_2(C)$ is not output by $g_2$, that is, $\neg f_2(\Box C)$ holds.
\end{proof}

\begin{cl}\label{cl:6-7}
For any formula $\varphi$, $T \vdash \PR_{g_2}^n (\gn{\varphi}) \to \PR_{g_2}^m (\gn{\varphi})$.
\end{cl}
\begin{proof}
We consider the following two cases.
\paragraph{Case 1:} $\varphi$ is not of the form $f_2(A)$ for any $A \in \mathsf{MF}$.

We work in $T$: By the definition of $\sigma(x)$, $\neg \Con_{\sigma}$ holds. Then, there exist some $i \neq 0$ and $s$
such that $h'(s)=0$ and $h'(s+1)=i$ by Proposition \ref{prop:h'}.2. Let $\xi_u \equiv \PR_{g_2}^{m-1}(\gn{\varphi})$.
Then, $\xi_u$ is not of the form $f_2(B)$ for any $B \in \mathsf{MF}$. 
Thus, we obtain $g_2(s+u)= \xi_u$. It follows that $\PR_{g_2}^m(\gn{\varphi})$ holds.

\paragraph{Case 2:} $\varphi$ is of the form $f_2(A)$ for some $A \in \mathsf{MF}$.

We further distinguish the following two cases.
\begin{itemize}
\item 
$\mathbf{N} \mathbf{A}_{m,n} \vdash \Box^n A$: 

Then, we obtain $\mathbf{N} \mathbf{A}_{m,n} \vdash \Box^m A$ by $\mathrm{Acc}_{m,n}$. This implies
\begin{equation}\label{4}
\PA \vdash \forall x \forall y (x \in W_y \to x \Vdash_y \Box^m A).
\end{equation}
We work in $T$: Since $\neg \Con_{\sigma}$ holds, we get $S'(\num{i})$ for some $i \neq 0$.
Let $i \in W_k$. Then,  $i \Vdash_k \Box^m A$  holds by (\ref{4}).
Hence, it follows from Claim \ref{cl:6-5} that $f_2(\Box^m A)$ holds, which means that $\PR_{g_2}^m (\gn{f_2(A)})$ holds.
\item 
$\mathbf{N} \mathbf{A}_{m,n} \nvdash \Box^n A$: 

Let $j$ be such that $\Box^n A \equiv A_j$ and $\bigl( W_j, \{ \prec_{j,B}\}_{B \in \MF}, \Vdash_j \bigr)$ be a countermodel of $A_j$.
Thus we obtain $r \nVdash_j \Box^n A$ for some $r \in W_j$, and $l \nVdash_j \Box^{n-1}A$ holds for some $l \in W_j$.
As in the proof of Claim \ref{cl:6-6}, we obtain $\PA \vdash \varphi_{\Box^{n-1}A}(\num{l}) \to \neg S'(\num{l})$.
From Claim \ref{cl:6-6}.2, it follows that $\PA \vdash \exists x \neg \varphi_{\Box^{n-1}A}(x) \to \neg f_2(\Box^{n-1}A)$.
Let $p$ and $q$ be $T$-proofs of  $\varphi_{\Box^{n-1}A}(\num{l}) \to \neg S'(\num{l})$ and $\exists x \neg \varphi_{\Box^{n-1}A}(x) \to \neg f_2(\Box^{n-1}A)$ respectively.
We prove $T \vdash \neg \PR_{g_2}^m (\gn{f_2(A)}) \to \neg \PR_{g_2}^n (\gn{f_2(A)})$.

We work in $T$: It follows that $\neg \Con_{\sigma}$ holds. 
Let $i \neq 0$ and let $s$ be such that $h'(s)=0$ and $h'(s+1)=i$.
Set $\xi_t \equiv f_2(\Box^{n-1} A)$ and $\xi_u \equiv f_2(\Box^{m-1} A)$ for $u<t$.
Suppose, towards a contradiction, that $ \neg \PR_{g_2}(\gn{f_2(\Box^{m-1}A)})$ and $ \PR_{g_2}(\gn{f_2(\Box^{n-1}A)})$ hold.
Then, $f_2(\Box^{n-1}A)$ is output by $g_2$.
Since $\xi_u$ is not output by $g_2$, we obtain $g_2(s+u)$=0. Thus, we get $g_2(s+t)=0$.
Therefore, $f_2(\Box^{n-1}A)$ is output by $g_2$ in Procedure 1, and $f_2(\Box^{n-1}A) \in P_{\sigma, s-1}$ holds.
Also, since $q \leq s-1$, we obtain $\exists x \neg \varphi_{\Box^{n-1}A}(x) \to \neg f_2(\Box^{n-1}A) \in P_{\sigma, s-1}$. 
Thus, $\forall x \varphi_{\Box^{n-1}A}(x)$ is a t.c.~of $P_{\sigma, s-1}$.
By $p \leq s-1$, $\varphi_{\Box^{n-1}A}(\num{l}) \to \neg S'(\num{l})$ is a t.c.~of $P_{\sigma, s-1}$.
Hence, it follows from $\Box^{n-1}A \in \Sub(A_j)$ and $l \in W_j$ that $h'(s) \neq 0$. This is a contradiction.
We conclude that $\neg \PR_{g_2}^m (\gn{f_2(A)})$ implies $\neg \PR_{g_2}^n (\gn{f_2(A)})$.
\end{itemize}
In either case, we have proved $T \vdash \PR_{g_2}^n (\gn{\varphi}) \to \PR_{g_2}^m (\gn{\varphi})$.
\end{proof}
We finish our proof of Theorem \ref{thm4-6}. 
Clause 1 and the implication $(\Rightarrow)$ of Clause 2 trivially hold by Claim \ref{cl:6-4} and Claim \ref{cl:6-7}.  

We prove the implication $(\Leftarrow)$ of Clause 2. Suppose $\mathbf{N} \mathbf{A}_{m,n} \nvdash A$. 
Then, there exists a $k \in \omega$ such that $A \equiv A_k$.
Let $\bigl( W_k, \{\prec_{k,B}\}_{B \in \MF}, \Vdash_k  \bigr)$ be a countermodel of $A_k$ and $j \in W_k$ be $j \nVdash_k A_k$.
Hence, we obtain 
\[
\PA \vdash \num{j} \neq 0 \wedge \exists y \bigl( \num{j} \in W_y \wedge A \in \Sub(A_y) \wedge \num{j} \nVdash_y A \bigr)
\]
and it follows that $\PA \vdash \varphi_A (\num{j}) \to \neg S'(\num{j})$.
By using the contrapositive of Claim \ref{cl:6-6}.2, we get
$\PA \vdash f_2(A)  \to \forall x \varphi_A (x)$, which implies $\PA \vdash f_2(A)  \to  \varphi_A (\num{j})$.
Then, we obtain $\PA \vdash f_2(A) \to \neg S'(\num{j})$.
From Proposition \ref{prop:h'}.3, $T \nvdash f_2(A)$ holds.
\end{proof}
%\begin{cor}
%Suppose $T$ is not $\Sigma_1$-sound and $n > m \geq 1$.
%There exists a $\Sigma_1$ provability predicate $\PR_T(x)$ such that
%$\NA_{m,n} = \PL(\PR_T)$.\end{cor}

\subsection{The case $m \geq 1$ and $n = 0$}
In this subsection, we verify the following theorem, which is Theorem \ref{thm4-3} with respect to the case $m > n = 0$.
%In fact, we have already given a proof of the case $m > n \geq 1$ in Theorem \ref{thm4-4}.
\begin{thm}\label{thm4-7}
	Suppose that $T$ is $\Sigma_1$-ill and $m \geq 1$. Then, there exists a $\Sigma_1$ provability
	predicate $\PR_T(x)$ of $T$ 
    such that
\begin{enumerate} 
\item 
for any $A \in \mathsf{MF}$ and arithmetical interpretation $f$ based on $\PR_T(x)$,
if $\mathbf{N}\mathbf{A}_{m,0} \vdash A$, then $T \vdash f(A)$, and 
\item 
there exists an arithmetical interpretation $f$ based on $\PR_T(x)$ such that 
$\mathbf{N}\mathbf{A}_{m,0} \vdash A$ if and only if $T \vdash f(A)$.
\end{enumerate}
\end{thm}
In fact, we can prove Theorem \ref{thm4-3} with respect to the more general case $m > n \geq 0$ in the same way as the proof of Theorem \ref{thm4-7}.  
%we study arithmetical completeness theorems of $\NA_{m,n}$ for $n > m \geq 1$.
%Especially, we prove that arithmetical completeness of $\NA_{0,n}$ holds for $\Sigma_1$-ill theories unlike $\Sigma_1$-sound theories.
%We verify the following uniform version of the arithmetical completeness theorem.
%\begin{thm}\label{thm:5-6}
%Suppose that $T$ is not $\Sigma_1$-sound and $m>n\geq 0$. There exists a $\Sigma_1$ provability
%predicate $\PR_T(x)$ of $T$ such that
%\begin{enumerate} 
%\item 
%for any $A \in \mathsf{MF}$ and arithmetical interpretation $f$ based on $\PR_T(x)$,
%if $\mathbf{N}^+\mathbf{A}_{m,n} \vdash A$, then $T \vdash f(A)$, and 
%\item 
%there exists an arithmetical interpretation $f$ based on $\PR_T(x)$ such that 
%$\mathbf{N}^+\mathbf{A}_{m,n} \vdash A$ if and only if $T \vdash f(A)$.
%\end{enumerate}
	
%\end{thm}
\begin{proof}
Let $m \geq 1$.
We prepare a primitive recursive enumeration of all $\NA_{m,0}$-unprovable formulas $\langle A_k  \rangle_{k \in \omega}$.
As in the proof of Theorem \ref{thm4-2}, for each $k \in \omega$, we can primitive recursively construct a finite $(m,0)$-accessible $\N$-model $\bigl( W_k, \{ \prec_{k,B}\}_{B \in \MF}, \Vdash_k   \bigr)$ which falsifies $A_k$
and let $\mathcal{M} = \bigl( W, \{\prec_{B}\}_{B \in \MF}, \Vdash \bigr)$ be the $(m,0)$-accessible $\N$-model defined as a disjoint union of these models.
We prepare the $\PA$-provably recursive function $h'$ constructed from the enumerations $\langle A_k  \rangle_{k \in \omega}$ and $\bigl \{ \bigl( W_k, \{ \prec_{k,B}\}_{B \in \MF}, \Vdash_k   \bigr) \bigr \}_{k \in \omega}$.

By the recursion theorem,
we define a $\PA$-provably recursive function $g_3$ corresponding to the case $m \geq 1$ and $n=0$.
In the definition of $g_3$, we use
the $\Sigma_1$-formula $\PR_{g_3}(x) \equiv$ $\exists y (g_3(y)=x \wedge \Fml(x))$
and the arithmetical interpretation $f_3$ based on $\PR_{g_3}(x)$ such that 
$f_3(p) \equiv \exists x (S'(x) \wedge x \neq 0 \wedge x \Vdash p)$.
%For each $r \in \omega$ and $\varphi$, $\PR_{g_3}^r(\gn{\varphi})$ is defined in the same way as $\PR_{g_0}^r(\gn{\varphi})$.
The function $g_3$ is defined as follows:
\medskip

\textbf{Procedure 1.}

Stage $s$:
\begin{itemize}
\item If $h'(s+1) =0$,
\begin{equation*}
  g_3(s)  = \begin{cases}
       \varphi & \text{if}\ s\ \text{is a}\ T_{\sigma} \text{-proof of}\ \varphi \\
               0 & \text{otherwise}.
             \end{cases}
  \end{equation*}

Then, go to Stage $s+1$.

\item If $h'(s+1) \neq 0$, go to Procedure 2.
\end{itemize}

\textbf{Procedure 2.}

Suppose $s$ and $i \neq 0$ satisfy $h'(s)=0$ and $h'(s+1)=i$. 
Let $k$ be a number such that $i \in W_k$. 
Define

\begin{equation*}
g_3(s+t)= \begin{cases} 0 & \text{if}\  \xi_t \equiv f_3(B)\  \&\ i \nVdash_k \Box B \ \text{for some}\ \Box B \in \mathsf{Sub}(A_k) \\
 
  \xi_t & \text{otherwise}.
		\end{cases}
\end{equation*}
We finish the construction of $g_3$.

\begin{cl}\label{cl:7-1}
For any formula $\varphi$,
$\PA + \Con_{\sigma} \vdash \PR_{g_3}(\gn{\varphi}) \leftrightarrow \PR_{\sigma}(\gn{\varphi})$.
\end{cl}
\begin{proof}
This is proved as in the proof of Claim \ref{cl:6-4}.
\end{proof}
Claim \ref{cl:7-1} guarantees that $\PR_{\sigma}(x)$ is a $\Sigma_1$ provability predicate.	
	
\begin{cl}\label{cl:7-2}
Let $B \in \MF$.
\begin{enumerate}
\item 
$\PA \vdash \exists x \bigl( x \neq 0 \wedge S'(x) \wedge \exists y(x \in W_y \wedge B \in \Sub(A_y) \wedge x \Vdash_y B)  \bigr) \to f_3(B)$.
\item 
$\PA \vdash \exists x \bigl( x \neq 0 \wedge S'(x) \wedge \exists y(x \in W_y \wedge B \in \Sub(A_y) \wedge x \nVdash_y B)  \bigr) \to \neg f_3(B)$.
\end{enumerate}
	
\end{cl}
\begin{proof}
By induction on the construction of $B \in \Sub(A_k)$, we prove clauses 1 and 2 simultaneously.
We only give a proof of the case $B \equiv \Box C$ for some $C \in \MF$.
\begin{enumerate}
\item 
We work in $\PA$:  Let $i \neq 0$ and $k$ be such that
$S'(\num{i})$, $i \in W_k$, $\Box C \in \Sub(A_k)$ and $i \Vdash_k \Box C$. Let $s'$ be such that $h'(s)=0$ and $h'(s+1) =i$.
Set $\xi_t \equiv f_3(C)$. 
Since $f_3$ is injective, we obtain $g_3(s+t)= \xi_t$, so $f_3(\Box C)$ holds.

\item 
This is verified as in the proof of clause 2 in Claim \ref{cl:6-6}.
\qedhere
\end{enumerate}
\end{proof}

\begin{cl}\label{cl:7-3}
For any formula $\varphi$, $T \vdash \varphi \to \PR_{g_3}^m (\gn{\varphi})$.
\end{cl}
\begin{proof}
We prove the contrapositive $T \vdash \neg \PR_{g_3}^m (\gn{\varphi}) \to \neg \varphi$.
We work in $T$:	Suppose $\neg \PR_{g_3}^m(\gn{\varphi})$ holds. Since we obtain $\PR_{\sigma}(\gn{0=1})$, there exists an $i \neq 0$ such that $S'(\num{i})$.
Let $k$ and $s$ be such that $i \in W_k$, $h'(s)=0$ and $h'(s+1)=i$. 
Set $\xi_t \equiv \PR_{g_3}^{m-1}(\gn{\varphi})$.
Since $\xi_t$ is not output by $g_3$ in Procedure 2, we get $g_3(s+t)=0$.
Hence, it follows that $\xi_t \equiv f_3(B)$ and $i \nVdash_k \Box B$ for some $\Box B \in \Sub(A_k)$.
Since $f_3(B) \equiv \PR_{g_3}^{m-1}(\gn{\varphi})$, there exists a $C \in \Sub(A_k)$ such that $B \equiv \Box^{m-1}C$ and $f_3(C) \equiv \varphi$.
Thus, we obtain $i \nVdash_k \Box^m C$, which implies $i \nVdash_k  C$ by $\mathrm{Acc}_{m,0}$.
It follows from clause 2 of Claim \ref{cl:7-2} that $\neg f_3(C)$ holds. Hence, we conclude $\neg \varphi$ holds.
\end{proof}
We complete our proof of Theorem \ref{thm4-7}. 
By Claim \ref{cl:7-3}, we obtain the first clause.
Also, the second clause is proved from Proposition \ref{prop:h'}.3 and Claim \ref{cl:7-2} as in the proof of Theorem \ref{thm4-6}.
\end{proof}
\section{Discussion}\label{sec6}
In this paper, we focused on the modal principle $\Box^n A \to \Box^m A$ and
we analyzed the logic $\NA_{m,n}$, which has this principle, from the viewpoint of provability logic.  
The corresponding arithmetical principle 
\[
\mathbf{D3}^n_m: T \vdash \PR_T^n(\gn{\varphi}) \to \PR_T^m(\gn{\varphi})
\]
 is known to be a meaningful condition from the viewpoint of the second incompleteness theorem. 
Kurahashi \cite{Kur25} investigated $\D{3}^n_m$ in the context of the second incompleteness theorem.
He also introduced the following conditions, which are weaker than $\D{2}$:
\begin{align*}
\mathbf{E}: & \ T \vdash \varphi \leftrightarrow \psi \Longrightarrow T \vdash \PR_T(\gn{\varphi}) \leftrightarrow \PR_T(\gn{\psi}), \\
\mathbf{C}: & \ T \vdash \PR_T(\gn{\varphi}) \wedge \PR_T(\gn{\psi}) \to \PR_T(\gn{\varphi \wedge \psi}).
\end{align*}
He proved that if $\PR_T(x)$ satisfies $\D{3}^n_m$ for some $m>n\geq 1$, $\mathbf{E}$, and $\mathbf{C}$, then $T \nvdash \neg \PR_T(\gn{0=1})$ holds, which refined the second incompleteness theorem.
Concerning $\D{3}^n_m$, we obtain the following as a corollary of clause 1 of Theorem \ref{main}.
\begin{cor}
For each $m,n \geq 1$,
\[
\NA_{m,n} = \bigcap \{\PL(\PR_T) \mid \PR_T(x)  \text{ is a } \Sigma_1 \text{ provability predicate satisfying } \D{3}^n_m \}.
\]
\end{cor}
Moreover, concerning the conditions $\mathbf{E}$ and $\mathbf{C}$, 
the following is obtained 
in the author’s recent work \cite{Kog}. Let $\mathbf{EN}$ be the logic obtained from $\mathbf{N}$ by adding the rule $\dfrac{A \leftrightarrow B}{\Box A \leftrightarrow \Box B}$.
Let $\mathbf{ECN}$ be the logic 
obtained from $\mathbf{N}$ by adding the axiom $\Box A \wedge \Box B \to \Box (A\wedge B)$.
\begin{thm}[\cite{Kog}]
\leavevmode

\begin{itemize}
    \item 
    $\mathbf{EN} = \bigcap \{\PL(\PR_T) \mid \PR_T(x)  \text{ is a } \Sigma_1 \text{ provability predicate satisfying } \mathbf{E} \}$.
    \item 
    $\mathbf{ECN} = \bigcap \{\PL(\PR_T) \mid \PR_T(x)  \text{ is a } \Sigma_1 \text{ provability predicate satisfying } \mathbf{E} \text{ and } \mathbf{C} \}$.
\end{itemize} 
\end{thm}
It is natural to ask whether the arithmetical completeness theorems still hold when the principle $\Box^n A \to \Box^m A$ is added to the logics $\mathbf{EN}$ and $\mathbf{ECN}$.
However, these seem to be difficult problems.
In fact, even the arithmetical completeness for the logics $\mathbf{EN4}$ and $\mathbf{ECN4}$ obtained by adding $\Box A \to \Box\Box A$ to $\mathbf{EN}$ and $\mathbf{ECN}$, respectively, is already hard to establish.

Also, let $\mathbf{CN}$ denote the logic obtained by adding the modal counterpart  $ \Box A \wedge \Box B \to \Box (A \wedge B)$ of $\mathbf{C}$ to $\mathbf{N}$.
It is not known whether the arithmetical completeness theorem holds for $\mathbf{CN}$.
However, to prove the arithmetical completeness theorem for $\mathbf{CN}$, the finite frame property of $\mathbf{CN}$ is required.
At present, it is not even known whether $\mathbf{CN}$ has the finite frame property with respect to $\mathbf{N}$-frames.
These issues concerning provability logics will be discussed in detail in a forthcoming paper \cite{KK3}.
% We expect that future research will continue to develop more fine-grained analysis of modal principles that are crucial for the second incompleteness theorem.

\section*{Acknowledgement}
The author would like to thank Taishi Kurahashi for his valuable comments.
\bibliography{namn_ref}
\bibliographystyle{plain}

\end{document}